\documentclass[12pt,a4paper]{article}
\usepackage{fullpage}
\usepackage[utf8]{inputenc}
\usepackage{amsfonts}
\usepackage{amsmath}
\usepackage{amssymb}
\usepackage{amsthm}
\usepackage{hyperref}
\usepackage[usenames]{color}
\usepackage{tikz}
\usetikzlibrary{calc,intersections,through,backgrounds,arrows,arrows.meta}

\usepackage[refpage]{nomencl}

\newtheorem{theorem}{Theorem}[section]
\newtheorem{lemma}[theorem]{Lemma}
\newtheorem{proposition}[theorem]{Proposition}
\newtheorem{statement}[theorem]{Statement}
\newtheorem{corollary}[theorem]{Corollary}

\theoremstyle{definition}
\newtheorem{definition}[theorem]{Definition}

\newtheorem{remark}[theorem]{Remark}
\newtheorem{notation}[theorem]{Notation}

\renewcommand\le{\leqslant}
\renewcommand\ge{\geqslant}

\def\bbZ{\mathbb Z}

\def\bbP{\mathbb P}

\def\mcA{\mathcal A}
\def\mcB{\mathcal B}
\def\mcC{\mathcal C}

\def\a{\mathfrak a}

\def\b{\mathfrak b}

\def\Id{\mathrm{Id}}

\def\One{\mathbf{1}}
\def\Lift{\Lambda}

\def\Lift{\mathrm{Lift}\,}

\def\SET{\mathrm{SET}}
\def\SEQ{\mathrm{SEQ}}
\def\CYC{\mathrm{CYC}}
\def\MSET{\mathrm{MSET}}
\def\PSET{\mathrm{PSET}}

\def\mcL{\mathcal L}       
\def\mcIL{\mathcal{IL}}    
\def\mcLM{\mathcal{LM}}    
\def\mcILM{\mathcal{ILM}}  
\def\l{\mathfrak l}        
\def\il{\mathfrak{il}}     
\def\lm{\mathfrak {lm}}    
\def\ilm{\mathfrak{ilm}}   

\def\mcP{\mathcal P}      
\def\mcIP{\mathcal{IP}}   
\def\p{\mathfrak p}       
\def\ip{\mathfrak{ip}}    

\def\mcM{\mathcal M}      
\def\mcIM{\mathcal{IM}}   
\def\m{\mathfrak m}       
\def\im{\mathfrak{im}}    

\def\lm{\mathfrak{lm}}       

\def\mcT{\mathcal T}      
\def\mcIT{\mathcal{IT}}   
\def\t{\mathfrak t}       
\def\it{\mathfrak{it}}    

\def\h{\mathfrak h}

\def\mcD{\mathcal D}      
\def\d{\mathfrak d}       

\author{Thierry Monteil\footnote{
    LIPN, CNRS (UMR 7030), Universit\'e Paris 13, F-93430 Villetaneuse, France.\newline
    Email: \texttt{thierry.monteil@lipn.univ-paris13.fr}}
    \and
    Khaydar Nurligareev\footnote{
    LIPN, CNRS (UMR 7030), Universit\'e Paris 13, F-93430 Villetaneuse, France.\newline
    LIB, Université de Bourgogne, F-21078, Dijon, France.\newline
    LIP6, CNRS (UMR 7606), Sorbonne Universit\'e, F-75252 Paris, France.\newline
    Email: \texttt{khaydar.nurligareev@lip6.fr}}
  }
\date{}

\title{Asymptotic probability of irreducibles II: sequence}

\begin{document}

\maketitle

\begin{abstract}
 This paper is devoted to the structure of the complete asymptotic expansion of the probability that a large combinatorial object is irreducible or consists of a~given number of irreducible parts, where irreducibility is understood in terms of combinatorial construction $\SEQ$, labeled or unlabeled.
 We show that for rapidly growing (\emph{i.e.} gargantuan) combinatorial classes, the coefficients that appear in this expansion are integers and can be interpreted as linear combinations of the counting sequences of three closely related combinatorial classes.
 We apply this general asymptotic result to labeled and unlabeled (multi-)tournaments, as well as to (multi-)permutations and (multi-)matchings.
 We also explore the limits of our approach with respect to other combinatorial constructions.
\end{abstract}

\section{Introduction}
\label{section: introduction}

In this paper, we continue our investigation devoted to studying the probability that large combinatorial objects are irreducible.

Our first work~\cite{MonteilNurligareevSET} in this series focuses on connectedness as a notion of irreducibility.
Its initial question is the following: how likely is a random object of size $n$ connected, as $n\to\infty$.
Various combinatorial objects show different behaviors.
Thus, for simple labeled graphs, the probability of connectedness tends to $1$~\cite{Gilbert1959}, for permutations understood as labeled directed graphs, it tends to 0, while for labeled forests it converges to~$1/\sqrt{e}$~\cite{Renyi1959}.
In our research~\cite{MonteilNurligareevSET}, we concentrate on the case where \emph{counting sequences} of the combinatorial objects under consideration \emph{are growing rapidly}
(the precise meaning of these words is captured by the notion of \emph{gargantuan sequence}; see Definition~\ref{def: gargantuan sequence}).
This restriction guaranties that the limiting probability is $1$~\cite{BellBenderCameronRichmond2000} and makes it possible to provide a complete asymptotic expansion~\cite{Bender1975}.
For example, for the probability $p_n$ that a~simple labeled graph of size $n$ is connected, the first three nontrivial terms indicated by Wright~\cite{Wright1970apr} are the following:
\begin{equation}\label{formula: Wright's asymptotic for graphs}
 p_n = 1 - \binom{n}{1}\frac{1}{2^{n-1}} - 2\binom{n}{3}\frac{1}{2^{3n-6}} - 24\binom{n}{4}\frac{1}{2^{4n-10}} + O\left(\frac{n^5}{2^{5n}}\right).
\end{equation}

The main contribution of the paper~\cite{MonteilNurligareevSET} is to provide the structure of the asymptotic expansion and to interpret its coefficients combinatorially.
Our method relies on identifying a \emph{double decomposition} of a labeled combinatorial class: on the one hand, as a set of connected components and, on the other hand, as a sequence of elements from a “derivative” class.
In terms of the symbolic method, this dual perspective can be expressed in the following way.
If a labeled combinatorial class~$\mcA$ admits a decomposition
 \[
  \mcA = \SET(\mcC) = \SEQ(\mcD)
 \]
and its counting sequence $(\a_n)$ is gargantuan, then the asymptotic probability that an element $a\in\mcA$ is connected satisfies
 \begin{equation}\label{formula: SET-asymptotics}
  \mathbb{P}(a\mbox{ is connected})
   \approx
   1 - 
  \sum\limits_{k\ge1}
   \d_k\cdot\binom{n}{k}\cdot \frac{\a_{n-k}}{\a_n},
 \end{equation}
where $(\d_k)$ is the counting sequence of the "derivative" class $\mcD$
(for the notion of \emph{asymptotic probability} and the exact meaning of the symbol $\approx$, see Definition~\ref{def: asymptotic probability} and Notation~\ref{notation: approx}, respectively).
For example, irreducible tournaments appear in expansion~\eqref{formula: Wright's asymptotic for graphs} for connected graphs, indecomposable permutations appear in the expansion for connected origamis, etc.
We refer to relation~\eqref{formula: SET-asymptotics} as \emph{$\SET$-asymptotics}.

In this paper, which is a natural extension of~\cite{MonteilNurligareevSET}, we explore the limits of the approach based on the symbolic method.
Here, the main focus is concentrated on the case where irreducibility is understood with respect to the labeled combinatorial construction $\SEQ$.
The reader might keep in mind irreducible tournaments as a motivating example.
The first estimation of the probability $q_n$ that a labeled tournament of size $n$ is irreducible was provided by Moon and Moser~\cite{MoonMoser1962} in~1962:
 \[
  \left|
   1 - \dfrac{2n}{2^{n-1}} - q_n
  \right| <
  \dfrac{1}{2^{n-1}}
  \qquad\mbox{for } n\ge14.
 \]
Several years later, Moon~\cite{Moon1968} improved on their result by obtaining the inequality
 \[
  \left|
   1 - \dfrac{2n}{2^{n-1}} - q_n
  \right| <
  \dfrac{1}{2}\left(\dfrac{n}{2^{n-1}}\right)^2.
 \]
In~1970, Wright~\cite{Wright1970jul} developed a recurrent method that allowed him to determine any predefined number of terms in the asymptotic probability~$q_n$ and computed its first five terms:
 \begin{equation}\label{formula: Wright's asymptotic for tournaments}
  q_n = 1 -
   \binom{n}{1}\dfrac{2^{2}}{2^{n}} +
   \binom{n}{2}\dfrac{2^{4}}{2^{2n}} -
   \binom{n}{3}\dfrac{2^{8}}{2^{3n}} -
   \binom{n}{4}\dfrac{2^{15}}{2^{4n}} +
   O\left(\dfrac{n^{5}}{2^{5n}}\right).
 \end{equation}
Similarly to expansion~\eqref{formula: Wright's asymptotic for graphs}, we would like to provide a structure of asymptotics~\eqref{formula: Wright's asymptotic for tournaments} and to interpret its coefficients combinatorially.

So we consider a labeled combinatorial class $\mcA$ represented as a sequence of another class, that is, $\mcA=\SEQ(\mcB)$.
In the spirit of the paper~\cite{MonteilNurligareevSET}, we assume that the counting sequence $(\a_n)$ of the class $\mcA$ is gargantuan.
Under these assumptions, we establish the complete asymptotic expansion of the probability that a random object $a\in\mcA$ is irreducible and, moreover, that $a$ is represented by a sequence of a given length $m$.
More precisely, the following result holds.

{
\renewcommand{\thetheorem}{\ref{theorem: SEQ_m-asymptotics}}
\begin{theorem}[$\SEQ$-asymptotics]
 Let $\mcA$ be a gargantuan labeled combinatorial class satisfying $\mcA = \SEQ(\mcB)$ for some labeled combinatorial class $\mcB$.
 Suppose that $a\in\mcA$ is a~random object of size~$n$.
 In this case, for any positive integer $m$,
 \begin{equation*}
  \bbP(a\mbox{ has }m\mbox{ }\SEQ\mbox{-irreducible components}) \approx
  \sum\limits_{k\ge0} \d_{k,m} \cdot
   \binom{n}{k} \cdot \dfrac{\a_{n-k}}{\a_n},
 \end{equation*}
 where
 \[
  \d_{k,m} = m\Big(\b_k^{(m-1)}-2\b_k^{(m)}+\b_k^{(m+1)}\Big).
 \]
 Here, $\big(\b_n^{(k)}\big)$ is the counting sequence of the combinatorial class $\SEQ_k(\mcB) = \mcB^k$ that consists of the $k$-sequences of $\mcB$-objects.
\end{theorem}
\addtocounter{theorem}{-1}
}

It is worth mentioning the nature of the ``derivative'' sequences appearing as asymptotic coefficients in the above theorems.
In the first case discussed in~\cite{MonteilNurligareevSET}, the link between the class $\mcC$ of $\SET$-irreducibles and its ``derivative'' class $\mcD$ corresponds symbolically to the relation
 \[
  \dfrac{\partial\log(z)}{\partial z} = \dfrac{1}{z}.
 \]
In particular, this relation explains why the $\SEQ$ decomposition appears as a necessary condition for a combinatorial interpretation.
On the other hand, when we are interested in $\SEQ$-irreducibles (see Definition~\ref{def: SEQ-irreducible objects}), the corresponding symbolic relation is
 \[
  \dfrac{\partial}{\partial z}\left(\dfrac{1}{z}\right) = -\dfrac{1}{z^2}.
 \]
Here, there is no need for any additional structure, since the ``derivative'' relies on the $\SEQ$ decomposition itself and relates to objects that admit a decomposition into two $\SEQ$-irreducible components.
Moreover, we can further iterate the latter relation,
 \[
  \dfrac{\partial}{\partial z}\left(\dfrac{1}{z^m}\right) = -\dfrac{m}{z^{m+1}},
 \]
and thus obtain ``derivative'' structures corresponding to objects that admit a decomposition into arbitrary number $m$ of $\SEQ$-irreducible components.

\

There are two noteworthy applications that illustrate both the method and its limitations well.
The first of them, the irreducible labeled tournaments already mentioned above, is notable for the fact that its coefficient sequence is not positive.
The presence of a negative coefficient leaves no room for one to interpret the whole sequence as a~counting sequence of any combinatorial class, only as a linear combination.
Another application, indecomposable permutations, is subtle for one more reason.
Although the definition of an indecomposable permutation of size $n$ is based on labels from~$1$ to~$n$, such a permutation cannot be considered a truly labeled object.
In fact, here the ground set $[n]=\{1,\ldots,n\}$ serves not only for labeling, but also for comparing the permutation atoms.
Due to this duality, the notion of an indecomposable permutation is not stable under relabeling.
In particular, it is not possible to construct neither the labeled product nor the sequence of the class of indecomposable permutations.

We present two approaches that permit us to establish the asymptotics of indecomposable permutations.
The first of them (which we call \emph{lift}) consists of endowing a~permutation of size $n$ with an additional linear order on the set $[n]$.
This allows us to split the role of the ground set, and thus obtain a functorial labeled class tractable by Theorem~\ref{theorem: SEQ_m-asymptotics}.
The second approach suggests that we formally consider the class of permutations to be unlabeled.
In this case, the decomposition of a permutation into a sequence of indecomposable ones is well-defined.
Now, to establish the desired asymptotics, we adapt our main result for the unlabeled construction $\SEQ$, and then apply it to (unlabeled) indecomposable permutations.

\

The structure of the paper is the following.
In Section~\ref{section: tools}, we present the tools that we use further, including combinatorial classes, generating functions, and Bender's theorem.
Most of them are not new, and an experienced reader can skip the corresponding parts of the text without loss.
Here, our contribution concerns the concept of gargantuan sequences and gargantuan classes introduced in our previous paper~\cite{MonteilNurligareevSET}, see Section~\ref{subsection: Gargantuan sequences}.
In Section~\ref{section: labeled SEQ-asymptotics}, we prove our main result, Theorem~\ref{theorem: SEQ_m-asymptotics}, and provide its combinatorial explanation.
We then apply this general result to labeled tournaments and labeled multitournaments.

The next three sections are devoted to various adaptations of our main result. In Section~\ref{section: Lift}, we introduce the lift operation, which allows us to treat subsets that are not stable under relabeling, and apply it to indecomposable permutations.
In Section~\ref{section: p-periodic counting sequences}, we adapt Theorem~\ref{theorem: SEQ_m-asymptotics} to labeled combinatorial classes whose counting sequences are $p$-periodic for some $p>1$.
We apply this adaptation, together with the lift operation, to indecomposable perfect matchings.
In Section~\ref{section: unlabeled SEQ-asymptotics}, we adapt our asymptotic theorem to unlabeled structures, which gives us a shorter way to obtain asymptotics of indecomposable permutations, indecomposable perfect matchings, and a number of their generalizations.
We also establish asymptotics of unlabeled irreducible tournaments.

Finally, in Section~\ref{section: Conclusion}, we discuss possible extensions of our method, including its applicability to other combinatorial constructions, both labeled and unlabeled, and some related questions.
We complete the paper with an appendix providing the list of numerical values of asymptotic coefficients discussed in the previous sections.

\section{Tools}
\label{section: tools}
 
\subsection{Asymptotic notation}
\label{subsection: Asymptotic notation}

We use the standard $O$-notation (see~\cite{Bender1974}): for a sequence $(a_n)$,
\begin{itemize}
 \item 
  $O(a_n)$ denotes the set of all sequences $(b_n)$ satisfying $\limsup\limits_{n\to\infty}|b_n/a_n|<\infty$,
 \item 
  $o(a_n)$ denotes the set of all sequences $(b_n)$ satisfying $\limsup\limits_{n\to\infty}|b_n/a_n|=0$.
\end{itemize}
 Thus, equations of the form $b_n=c_n+O(a_n)$ are interpreted as $(b_n-c_n)\in O(a_n)$.

\begin{notation}\label{notation: approx} 
 For a sequence $(a_n)$ and an integer $m$, we write
  \[
   a_n \approx \sum\limits_{k\ge m}c_kf_k(n),
  \]
  if for every $r\ge m$,
  \[
   a_n = \sum\limits_{k=m}^{r}c_kf_k(n) + O\big(f_{r+1}(n)\big),
  \]
  and for every $k\ge m$,
  \[
   f_{k+1}(n) = o\big(f_k(n)\big).
  \]
Note that the sequence $(c_k)$ of constants may contain zeros.
\end{notation}

\subsection{Combinatorial classes and decompositions}
\label{subsection: combinatorial classes}

Compared to our previous paper~\cite{MonteilNurligareevSET}, here we employ both labeled and unlabeled combinatorial classes (or structures).
Both are pretty standard and can be found in textbooks, such as~\cite{Stanley1999, FlajoletSedgewick2009, Bona2015}.
However, we recall all the necessary information to emphasize the content that will be used in the following.

\subsubsection*{Combinatorial classes}

By \emph{combinatorial class} $\mcA$, we understand a collection of (combinatorial) objects of finite size such that the number $\a_n$ of objects of size $n$ is finite for any $n\ge0$.
The class $\mcA$ is \emph{labeled} if the ``atoms'' of every object of size $n$, such as graph vertices, are labeled by the elements of the set $[n]=\{1,\ldots,n\}$ (or any other set of cardinality $n$), and each label appears exactly once.
Otherwise, the class is \emph{unlabeled}.

By default, the combinatorial classes considered in this paper are assumed to be labeled.
To underline that a class under consideration is unlabeled, we mark it with a~tilde: $\widetilde{\mcA}$.
For associated concepts, such as counting sequences and generating functions, we use tildes as well.

\subsubsection*{Associated concepts}

For any combinatorial class~$\mcA$, we associate a number of additional concepts.
The first of them is the \emph{counting sequence}~$(\a_n)$, where $\a_n$ is the number of $\mcA$-objects of size $n$.
Another, more refined way of counting employs the \emph{generating function} $A(z)$, which is \emph{exponential} for labeled classes and \emph{ordinary} for unlabeled ones, that is,
 \[
  A(z)=\sum\limits_{n=0}^{\infty}\a_n\dfrac{z^n}{n!}
  \qquad\mbox{and}\qquad
  \widetilde{A}(z)=\sum\limits_{n=0}^{\infty}\a_n{z^n},
 \]
respectively.

One more structure associated with a combinatorial class $\mcA$ is its \emph{restriction} $\mcA_n$ to cardinality $n$, \emph{i.e.} the subclass consisting of all $\mcA$-objects of size~$n$.
If $\a_n\ne0$, we assume that this subclass is endowed with a uniform probability $\bbP_n$: each object $a\in\mcA_n$ appears with probability~$1/\a_n$.

\begin{definition}\label{def: asymptotic probability}
 If $Q$ is some property on the class~$\mcA$, then we denote by
 \(
  \bbP\big(a\mbox{ satisfies } Q\big)
 \)
 the sequence of probabilities
 \[
  \bbP_n\big(a\in\mcA_n\mid a\mbox{ satisfies } Q\big)_{\a_n>0},
 \]
 which we call the \emph{asymptotic probability} that a random object $a\in\mcA$ satisfies $Q$.
\end{definition}

\subsubsection*{Binary operations}

Given two combinatorial classes $\mcA$ and $\mcB$ (labeled or unlabeled), one can combine them to construct a new one.
In this paper, the following operations will be used for this purpose.
\begin{enumerate}
 \item
  The \emph{disjoint union} $\mcA+\mcB$ consisting of objects from both $\mcA$ and $\mcB$.
  Its generating function is $A(z) + B(z)$.
 \item
  For unlabeled combinatorial classes, the \emph{Cartesian product} $\widetilde{\mcA}\times\widetilde{\mcB}$ consisting of all possible ordered pairs $(a,b)$, where $a\in\widetilde{\mcA}$ and $b\in\widetilde{\mcB}$.
  Its ordinary generating function is $\widetilde{A}(z)\widetilde{B}(z)$.
 \item
  For labeled combinatorial classes, the \emph{labeled product} $\mcA\star\mathcal{B}$ consisting of all ordered pairs from $\mcA\times\mcB$ relabeled in all possible order-consistent manner.
  Its exponential generating function is $A(z)B(z)$.
 \item
  For labeled combinatorial classes, the \emph{Hadamard product} $\mcA\odot\mcB$ consisting of all possible ordered pairs $(a,b)\in\mcA\times\mcB$ of the same size (thus, the ``atoms'' of $a$ and~$b$ bearing the same labels are identified).
  Its exponential generating function is
  \[
   A(z) \odot B(z)
    = 
   \sum_{n=0}^{\infty} \a_n\b_n \dfrac{z^n}{n!}. 
  \]
\end{enumerate}

\subsubsection*{Decompositions}

Given a combinatorial class $\mcA$ (labeled or unlabeled) and an integer $m\in\bbZ_{\ge0}$, one can construct a new class by combining $m$ copies of $\mcA$.
This can be done in various ways; in this paper, we consider the following ones.
\begin{enumerate}
 \item
  The class $\SEQ_m(\mcA)=\mcA^m$ of \emph{$m$-sequences} of $\mcA$-objects.
  Here, by $\mcA^m$ we understand the labeled product of $m$ copies of the class $\mcA$ in the labeled case, and their Cartesian product in the unlabeled case, respectively.
  The corresponding generating function is $A^m(z)$ for the labeled case and $\widetilde{A}^m(z)$ for the unlabeled case.
 \item
  The class $\CYC_m(\mcA)$ of \emph{$m$-cycles} of $\mcA$-objects.
  Its generating function in the labeled and unlabeled cases is, respectively,
  \begin{equation}\label{formula: GF-CYC_m}
   \dfrac{A^m(z)}{m}
   \qquad\mbox{and}\qquad
   [u^m]\sum\limits_{n=1}^{\infty}\dfrac{\phi(n)}{n}\log\dfrac{1}{1-u^n\widetilde{A}(z^n)},
  \end{equation}
  where $\phi$ is the Euler totient function and the operator $[u^m]$ extracts the $m$th coefficient of a formal power series in $u$.
 \item
  In the labeled case, the class $\SET_m(\mcA)$ of \emph{$m$-sets} of $\mcA$-objects.
  Its exponential generating function is $A^m(z)/m!$.
 \item 
  In the unlabeled case, the classes $\MSET_m(\widetilde{\mcA})$ and $\PSET_m(\widetilde{\mcA})$ of \emph{$m$-multisets} and \emph{$m$-powersets} of $\widetilde{\mcA}$-objects, respectively (depending on whether identical copies of components are allowed or not).
  Their ordinary generating functions are, respectively,
  \begin{equation}\label{formula: GF-MSET_m}
   [u^m]\exp\left(\sum\limits_{n=1}^{\infty}\dfrac{u^n}{n}\widetilde{A}(z^n)\right)
   \qquad\mbox{and}\qquad
   [u^m]\exp\left(\sum\limits_{n=1}^{\infty}(-1)^{n-1}\dfrac{u^n}{n}\widetilde{A}(z^n)\right).
  \end{equation}
\end{enumerate}

If, additionally, $\a_0=0$, we can define the combinatorial class of \emph{sequences} by taking the disjoint union of $m$-sequences over all $m\in\bbZ_{\ge0}$, so that
 \[
  \SEQ(\mcA) =
  \sum\limits_{m=0}^{\infty}\SEQ_m(\mcA).
 \]
Similarly, the classes $\CYC$ of \emph{cycles}, $\SET$ of \emph{sets}, $\MSET$ of \emph{multisets}, and $\PSET$ of \emph{powersets} are defined.
The corresponding generating functions are indicated in Table~\ref{table: generating functions}.

 \begin{table}[ht!]
  \centering
  \begin{tabular}{|c|c||c|c|}
   \hline
    \multicolumn{2}{|c||}{labeled constructions} & \multicolumn{2}{|c|}{unlabeled constructions} \\
   \hline
   \hline
    $\SEQ(\mcA)$ & $\dfrac{1}{1-A(z)}$ & $\SEQ(\widetilde{\mcA})$ & $\dfrac{1}{1-\widetilde{A}(z)}$ \\
   \hline
    $\CYC(\mcA)$ & $\log\dfrac{1}{1-A(z)}$ & $\CYC(\widetilde{\mcA})$ & $\sum\limits_{n=1}^{\infty}\dfrac{\phi(n)}{n}\log\dfrac{1}{1-\widetilde{A}(z^n)}$ \\
   \hline
    $\SET(\mcA)$ & $\exp\big(A(z)\big)$ & $\MSET(\widetilde{\mcA})$ & $\exp\left(\sum\limits_{n=1}^{\infty}\dfrac{\widetilde{A}(z^n)}{n}\right)$ \\
   \hline
     &  & $\PSET(\widetilde{\mcA})$ & $\exp\left(\sum\limits_{n=1}^{\infty}\dfrac{(-1)^{n-1}\widetilde{A}(z^n)}{n}\right)$ \\
   \hline
  \end{tabular}
  \caption{Generating functions of some combinatorial constructions.}
  \label{table: generating functions}
 \end{table}

\begin{definition}\label{def: SEQ-irreducible objects}
 Let $\mcA$ and $\mcB$ be two combinatorial classes (labeled or unlabeled) that satisfy $\mcA=\SEQ(\mcB)$.
 An object $a\in\mcA$ is said to be \emph{$\SEQ$-irreducible} if $a\in\mcB$. In the case where $a\in\mcB^m$ for some positive integer $m$, we will say that $a$ has exactly $m$ \emph{$\SEQ$-irreducible components} (or \emph{$\SEQ$-irreducible parts}).

 In the same way, we can define \emph{$\CYC$-irreducible} and \emph{$\SET$-irreducible} objects.
\end{definition}

\subsection{Gargantuan classes and Bender's theorem}
\label{subsection: Gargantuan sequences}

Here, we recall the concepts of gargantuan sequences and gargantuan combinatorial classes introduced in~\cite{MonteilNurligareevSET}.
We also recall, in the simplified form of~\cite{Odlyzko1995}, Bender's theorem~\cite{Bender1975}, which is our main analytical tool to establish complete asymptotic expansions.

\begin{definition}\label{def: gargantuan sequence}
 A sequence $(a_n)$ is \emph{gargantuan} if, for any positive integer $r$, as $n\to\infty$, the following two conditions hold:
 \[
  \mbox{ (i) }\quad
  \dfrac{a_{n-1}}{a_n} \to 0;
  \qquad\qquad
  \mbox{ (ii) }\quad
  \sum\limits_{k=r}^{n-r}|a_ka_{n-k}| = O\big(a_{n-r}\big).
 \]

 An unlabeled combinatorial class $\widetilde{\mcA}$ is \emph{gargantuan} if its counting sequence $(\widetilde{\a}_n)$ is gargantuan.
 A labeled combinatorial class $\mcA$ is \emph{gargantuan} if the sequence $(\a_n/n!)$ is gargantuan.
\end{definition}

\begin{theorem}[Bender~\cite{Bender1975}]\label{theorem: Bender's}
 Consider the formal power series
 \[
  U(z) = \sum\limits_{n=1}^{\infty}u_nz^n
 \]
 and a function $F(x)$ analytic in a neighborhood of origin.
 Define
 \[
  V(z) = \sum\limits_{n=0}^{\infty}v_nz^n = F\big(U(z)\big)
 \qquad
 \mbox{and}
 \qquad
  W(z) = \sum\limits_{n=0}^{\infty}w_n z^n =
  \left.\dfrac{\partial}{\partial x} F(x)\right|_{x=U(z)}.
 \]
 Assume that $u_n\ne 0$ for all sufficiently large $n$, and that the sequence $(u_n)$ is gargantuan.
 In this case,
 \[
  v_n \approx \sum\limits_{k\ge0}w_ku_{n-k}
 \]
 and the sequence $(v_n)$ is gargantuan.
\end{theorem}

\begin{lemma}[Lemma~2.4 in~\cite{MonteilNurligareevSET}]\label{lemma: sufficient conditions for gargantuan sequence}
 If a sequence $(a_n)$ satisfies the following two conditions
 \begin{align*}\label{equation: sufficient conditions for gargantuan sequence}
  \emph{ (i)' }\quad &
  na_{n-1} = O(a_n),\,\,\mbox{ as }\, n\to\infty; \\
  \emph{ (ii)' }\quad &
  x_k = |a_ka_{n-k}| \mbox{ is decreasing for } k < n/2 \mbox{ and for all but finitely many } n,
 \end{align*}
 then $(a_n)$ is gargantuan.
\end{lemma}

\begin{lemma}[Lemma~2.5 in~\cite{MonteilNurligareevSET}]\label{lemma: a_nb_n is gargantuan}
 If $(a_n)$ and $(b_n)$ are two nonnegative gargantuan sequences, then the sequence $(a_nb_n)$ is nonnegative gargantuan as well.
\end{lemma}

\begin{lemma}\label{lemma: Ka_n+b_n is gargantuan}
 If $(a_n)$ is a nonnegative gargantuan sequence and $b_n=o(a_n)$, then, for any $K>0$, the sequence $c_n = Ka_n + b_n$ is gargantuan.
\end{lemma}
\begin{proof}
 Check the conditions of Definition~\ref{def: gargantuan sequence}.
 For the first condition, we have
 \[
  \dfrac{c_{n-1}}{c_n} =
  \dfrac{Ka_{n-1} + b_{n-1}}{Ka_n + b_n} = 
  \dfrac{a_{n-1}}{a_n}
   \left(1+\dfrac{b_{n-1}}{Ka_{n-1}}\right)
    \left(1+\dfrac{b_n}{Ka_n}\right)^{-1} \to 0,
 \]
 as $n\to\infty$.
 To verify the second condition, we note that there exists a constant $M>0$ such that $|b_n|<Ma_n$ for all $n\in\bbZ_{>0}$.
 Therefore, $|c_n|\le(K+M)a_n$, and we have
 \[
  \sum\limits_{k=r}^{n-r}|c_kc_{n-k}| \le 
  (K+M)^2\sum\limits_{k=r}^{n-r}a_ka_{n-k} =
  O(a_{n-r}) = O(c_{n-r}).
 \]
 Since both conditions hold, the sequence $(c_n)$ is gargantuan.
\end{proof}

\section{Asymptotics of the labeled construction SEQ}
\label{section: labeled SEQ-asymptotics}

This section is devoted to our main result concerning the asymptotic probability of labeled combinatorial objects that admit $\SEQ$ decomposition.
This general result was motivated by a specific application that first appeared in~\cite{MonteilNurligareev2021}, namely the asymptotic behavior of irreducible labeled tournaments.
We begin our exposition directly by establishing the complete asymptotic expansion of the probabilities under consideration.
Next, we provide a combinatorial explanation of the coefficients involved in the expansion.
We then apply this general theorem to labeled tournaments, and thus we obtain the asymptotic probability that a random tournament consists of a given number of irreducible components.
Finally, in the end of the section, we discuss higher dimension applications. 

\subsection{Asymptotic probability of labeled SEQ-irreducibles}
\label{subsection: labeled SEQ-asymptotics}

Let $\mcB$ be a labeled combinatorial class and $m$ be a positive integer.
Recall that by $\big(\b_n^{(m)}\big)$ we denote the counting sequence of the class $\SEQ_m(\mcB)=\mcB^m$.
By convention, we assume that $\b_{0}^{(0)} = 1$ and that $\b_{n}^{(0)} = 0$ for any $n>0$.

\begin{theorem}\label{theorem: SEQ_m-asymptotics}[$\SEQ$-asymptotics]
 Let $\mcA$ be a gargantuan labeled combinatorial class satisfying $\mcA = \SEQ(\mcB)$ for some labeled combinatorial class $\mcB$.
 Suppose that $a\in\mcA$ is a~random object of size~$n$.
 In this case, for any positive integer $m$,
 \begin{equation}\label{formula: SEQ_m-asymptotics}
  \bbP(a\mbox{ has }m\mbox{ }\SEQ\mbox{-irreducible components}) \approx
  \sum\limits_{k\ge0} \d_{k,m} \cdot
   \binom{n}{k} \cdot \dfrac{\a_{n-k}}{\a_n},
 \end{equation}
 where
 \[
  \d_{k,m} = m\Big(\b_k^{(m-1)}-2\b_k^{(m)}+\b_k^{(m+1)}\Big).
 \]
 In particular, for the case where $m=1$, we have
 \begin{equation}\label{formula: SEQ-asymptotics}
  \bbP(a\mbox{ is }\SEQ\mbox{-irreducible}) \approx
  1 - \sum\limits_{k\ge1}
   \Big(2\b_k-\b_k^{(2)}\Big) \cdot
   \binom{n}{k} \cdot \dfrac{\a_{n-k}}{\a_n}.
 \end{equation}
\end{theorem}
\begin{proof}
 For a fixed positive integer $m$, let us apply Theorem~\ref{theorem: Bender's} to the formal power series $U(z) = A(z) - 1$ and the function
 \[
  F(x) = \left(1-\dfrac{1}{1+x}\right)^m.
 \]
 The class $\mcA$ is the sequence of the class $\mcB$.
 Hence, their exponential generating functions satisfy $A(z)=\big(1-B(z)\big)^{-1}$, which implies that our target series is
 \[
  V(z) = F\big(A(z)-1\big) =
  \left(1-\dfrac{1}{A(z)}\right)^m = B^m(z).
 \] 
 On the other hand, the coefficients $\d_{k,m}$ come from the series
 \begin{align*}
  W(z) & =
  \left.\dfrac{\partial}{\partial x} F(x)\right|_{x=U(z)}\\ & =
  \left(1-\dfrac{1}{A(z)}\right)^{m-1}\dfrac{m}{\big(A(z)\big)^2}\\
  & = m\Big(B^{m-1}(z) - 2B^{m}(z) + B^{m+1}(z)\Big).
 \end{align*}
 Therefore, Theorem~\ref{theorem: Bender's} leads to the asymptotic relation
 \[
  \dfrac{\b_n^{(m)}}{n!} \approx
  \sum\limits_{k\ge0}
   \dfrac{\d_{k,m}}{k!} \cdot \frac{\a_{n-k}}{(n-k)!} =
  \dfrac{1}{n!} \sum\limits_{k\ge0}
   \binom{n}{k} \d_{k,m} \a_{n-k}.
 \]
 To complete the proof, we divide the two sides of this relation by $\a_n/n!$ and note that $\b_n^{(m)}/\a_n$ is the probability that a random object $a\in\mcA$ has $m$ $\SEQ$-irreducible components.
\end{proof}

\begin{corollary}\label{corollary: leading term of SEQ_m-asymptotics}
 If $\a_1\neq0$, then the leading term of asymptotic expansion~\eqref{formula: SEQ_m-asymptotics} satisfies
 \[
  \bbP(a\mbox{ has }m\mbox{ }\SEQ\mbox{-irreducible components}) = 
  m\cdot (n)_{m-1}\cdot
   \dfrac{\a_1^{m-1}\a_{n-m+1}}{\a_n} +
  O\left( n^{m}\cdot \dfrac{\a_{n-m}}{\a_n} \right),
 \]
 where $(n)_{m}=n(n-1)(n-2)\ldots(n-m+1)$ are the falling factorials.
\end{corollary}
\begin{proof}
 Due to Theorem~\ref{theorem: SEQ_m-asymptotics}, the object of our interest, for a fixed $m$, is the first nonzero coefficient in the sequence $(\d_{k,m})$,
 that is, the leading term of the exponential generating function $m\Big(B^{m-1}(z) - 2B^{m}(z) + B^{m+1}(z)\Big)$.
 Since $\b_0=0$ and $\b_1=\a_1>0$, we have
 \[
  m\Big(B^{m-1}(z) - 2B^{m}(z) + B^{m+1}(z)\Big) =
  m!\cdot \a_1^{m-1}\cdot\dfrac{z^{m-1}}{(m-1)!} + O(z^m).
 \]
 Therefore, the dominant term in asymptotics~\eqref{formula: SEQ_m-asymptotics} is
 \[
  m! \cdot \a_1^{m-1}
   \left(
    \binom{n}{m-1} \cdot \dfrac{\a_{n-m+1}}{\a_n}
   \right) = 
  m \cdot (n)_{m-1} \cdot
  \dfrac{\a_1^{m-1}\a_{n-m+1}}{\a_n}.
 \]
\end{proof}

\begin{remark}\label{remark: sums of d_(k,m)}
 As we can see from Theorem~\ref{theorem: SEQ_m-asymptotics}, a large object from a gargantuan class that admits a $\SEQ$ decomposition is most probably $\SEQ$-irreducible.
 Moreover, Corollary~\ref{corollary: leading term of SEQ_m-asymptotics} testifies that for each $r\in\mathbb{Z}_{>0}$, the first $r$ asymptotic coefficients in expansion~\eqref{formula: SEQ_m-asymptotics} can be nonzero only for $m\le r$.
 Another side of this behavior is represented by the fact that for every fixed $k>0$,
 \[
  \sum\limits_{m=1}^{\infty}\d_{k,m} = 0,
 \]
 which follows from the relation
 \[
  \sum\limits_{m=1}^{\infty}
  	m\Big(B^{m-1}(z)-2B^{m}(z)+B^{m+1}(z)\Big) =
  B^{0}(z) = 1.
 \]
 Here, the property of being gargantuan is crucial.
 Thus, the labeled combinatorial class of linear orders is not gargantuan, and the number of $\SEQ$-irreducible components in a~typical large linear order is large.
 The latter follows from the observation that the number of components of a linear order in fact coincides with its size.
\end{remark}

\subsection{Combinatorial interpretation of the asymptotic theorem}
\label{subsection: asymptotics interpretation}

The specific form of the coefficients in asymptotic expansion~\eqref{formula: SEQ_m-asymptotics} can be derived from the structure of particular objects of the labeled combinatorial class~$\mcA$.
Roughly speaking, the coefficient $\d_{k,m}$ is determined by objects of a large size $n$ that possess one \emph{large} and $(m-1)$~\emph{small} components where the total size of the small components does not exceed~$k$.
The goal of this section is to establish the exact form of this dependency and to present a satisfactory explanation.
Our main tool will be the inclusion-exclusion principle.

Let us start with the simplest case $m=1$.

\begin{statement}\label{lemma: SEQ-irreducible recurrence-2}
 If two labeled combinatorial classes $\mcA$ and $\mcB$ satisfy $\mcA=\SEQ(\mcB)$, then for any positive integer $n$,
 \begin{equation}\label{formula: SEQ-irreducible recurrence-2}
  \b_n
   =
  \a_n
    - 
   2\sum\limits_{k=1}^{[n/2]}
    \binom{n}{k} \b_k \a_{n-k}
    +
   \sum\limits_{p=1}^{[n/2]}
    \sum\limits_{q=1}^{[n/2]}
     \binom{n}{p,q} \b_p \b_q \a_{n-p-q}.
 \end{equation}
\end{statement}
\begin{proof}
 First, note that any object $a\in\mcA$ has at most one irreducible component whose size is greater than $n/2$.
 In particular, if $a$ is reducible, that is, if this object is represented as a nontrivial sequence of irreducible components, then the size of at least one of its first and last components is smaller than or equal to $n/2$.
 Therefore, to count the number of irreducible objects $\b_n$, we can use the inclusion-exclusion principle in the following way.
 \begin{enumerate}
  \item
   Take the total number $\a_n$ of objects of size $n$ in $\mcA$ (left side of Fig.~\ref{picture: counting SEQ-irreducible objects}).
  \item
   Deduce those of them whose first or last component is of size less than or equal to~$n/2$.
   There are $\binom{n}{k}\b_k\a_{n-k}$ such objects whose first irreducible component is of size $k\le n/2$, and the same number of objects with the last irreducible component of size $k$ (they are schematically depicted in the middle of Fig.~\ref{picture: counting SEQ-irreducible objects}).
   Thus, we deduce $2\binom{n}{k}\b_k\a_{n-k}$ for every $k$ satisfying $1\le k\le n/2$. 
  \item
   However, in the second step, a number of objects are deduced twice.
   Specifically, those are objects whose first and last components are less than or equal to $n/2$ in size.
   For fixed sizes $p,q\le n/2$ of these components, we have $\binom{n}{p,q}\b_p\b_q\a_{n-p-q}$ of such objects in total (right side of Fig.~\ref{picture: counting SEQ-irreducible objects}).
   Adding them to our difference, we obtain relation~\eqref{formula: SEQ-irreducible recurrence-2}.
 \end{enumerate}
\end{proof}

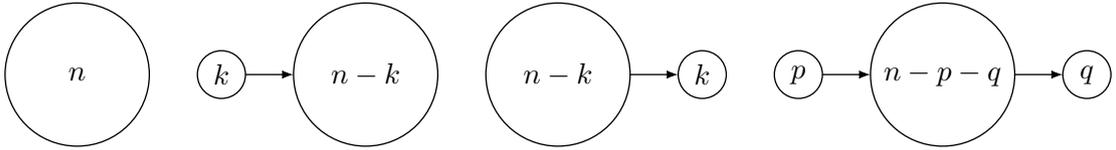
\begin{figure}[!ht]
\begin{center}
\begin{tikzpicture}[>= latex, line width=.5pt]
 \begin{scope}[scale=0.9]
  \small
  \def\quot{0.1}  
  \def\R{30pt}
  \def\r{10pt}
  \def\h{60pt}
  \def\d{20pt}
  \coordinate (c7) at (-\R-\r-\d,0);
  \coordinate (c0) at (0,0);
  \coordinate (c1) at (\h,0);
  \coordinate (c2) at (\h+2*\R+\d,0);
  \coordinate (c3) at (2*\h+2*\R+\d,0);
  \coordinate (c4) at (2*\h+2*\R+2*\r+2*\d,0);
  \coordinate (c5) at (3*\h+2*\R+2*\r+2*\d,0);
  \coordinate (c6) at (4*\h+2*\R+2*\r+2*\d,0);
  \foreach \p in {c0,c3}{
   \draw (\p) circle (\r);
   \draw (\p) node {$k$};
  }
  \foreach \p in {c4}{
   \draw (\p) circle (\r);
   \draw (\p) node {$p$};
  }
  \foreach \p in {c6}{
   \draw (\p) circle (\r);
   \draw (\p) node {$q$};
  }
  \foreach \p in {c1,c2}{
   \draw (\p) circle (\R);
   \draw (\p) node {$n-k$};
  }
  \foreach \p in {c5}{
   \draw (\p) circle (\R);
   \draw (\p) node {$n-p-q$};
  }
  \foreach \p in {c7}{
   \draw (\p) circle (\R);
   \draw (\p) node {$n$};
  }
  \foreach \p in {c0,c4} \draw[->] (\p)++(\r,0) -- ++(\h-\R-\r,0);
  \foreach \p in {c2,c5} \draw[->] (\p)++(\R,0) -- ++(\h-\R-\r,0);
 \end{scope}
\end{tikzpicture}
\end{center}
\caption{Schema for counting objects in $\mcA_n$.}\label{picture: counting SEQ-irreducible objects}
\end{figure}

For a fixed positive integer $r<n$, the last summand of the right side of relation~\eqref{formula: SEQ-irreducible recurrence-2} can be written as
 \[
  \sum\limits_{p=1}^{[n/2]}
   \sum\limits_{q=1}^{[n/2]}
    \binom{n}{p,q} \b_p \b_q \a_{n-p-q}
   =
  \sum\limits_{k=1}^{r-1}
   \binom{n}{k} \b_k^{(2)} \a_{n-k}
   +
  \sum\limits_{r\le p+q\le n}
   \binom{n}{p,q} \b_p \b_q \a_{n-p-q}.
 \]
In the case where the class $\mcA$ is gargantuan, the second summand is negligible, as $n\to\infty$.
Thus, from Statement~\ref{lemma: SEQ-irreducible recurrence-2} we obtain a combinatorial explanation of the coefficients of asymptotic expansion~\eqref{formula: SEQ-asymptotics}.
Namely, the asymptotic coefficients represent a~possible structure of the first and last irreducible components in the case where the sizes of these components are sufficiently small.
For example, to obtain the first two terms in asymptotics~\eqref{formula: SEQ-asymptotics}, it suffices to take into account those objects that begin or end with an irreducible component of size $1$.

\

A similar approach based on the inclusion-exclusion principle can be proposed in the general case as well.
We are not going to provide excessive details of the reasoning or an exact expression for $\b_n^{(m)}$.
Here are just a few key ideas.

First, we can count objects with respect to size.
In our case, this means that for an object of size $n$ with $m$ irreducible components, we identify the largest one (if there are several largest components, we take the first of them).
Depending on the position of the largest component, there could be $m$ possible configurations.
For example, the configurations corresponding to $m=3$ are depicted in Fig.~\ref{picture: m-sequences}.

\begin{figure}[!ht]
\begin{center}
\begin{tikzpicture}[>= latex, line width=.5pt]
 \begin{scope}[scale=0.9]
  \small
  \def\quot{0.1}  
  \def\R{30pt} 
  \def\r{10pt} 
  \def\h{60pt} 
  \def\d{80pt} 
  \coordinate (c0) at (-\h+\R-\r,0);
  \coordinate (c1) at (0,0);
  \coordinate (c2) at (\h,0);
  \coordinate (c3) at (\h+\d,0);
  \coordinate (c4) at (2*\h+\d,0);
  \coordinate (c5) at (3*\h+\d,0);
  \coordinate (c6) at (3*\h+2*\d,0);
  \coordinate (c7) at (4*\h+2*\d,0);
  \coordinate (c8) at (5*\h-\R+\r+2*\d,0);
  \foreach \p in {c0,c1,c3,c5,c7,c8}{
   \draw (\p) circle (\r);
  }
  \foreach \p in {c2,c4,c6}{
   \draw (\p) circle (\R);
  }
  \foreach \p in {c0,c1,c3,c7} \draw[->] (\p)++(\r,0) -- ++(\h-\R-\r,0);
  \foreach \p in {c4,c6} \draw[->] (\p)++(\R,0) -- ++(\h-\R-\r,0);
 \end{scope}
\end{tikzpicture}
\end{center}
\caption{Sequences with two small and one large components.}\label{picture: m-sequences}
\end{figure}
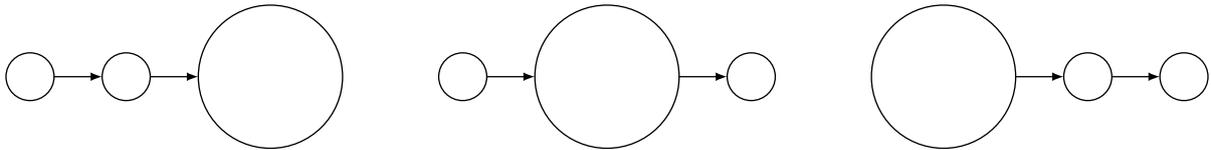

Next, we use the inclusion-exclusion principle with respect to the largest irreducible component.
In terms of counting, there are $\binom{n}{k_1,\ldots,k_m}\b_{k_1}\ldots\b_{k_m}$ objects of a given configuration, where $k_1,\ldots,k_m$ are the sizes of the components.
Assuming without loss of generality that $k_m$ is the size of the largest component, we replace $\b_{k_m}$ by the right side of relation~\eqref{formula: SEQ-irreducible recurrence-2}.
Combinatorially, this means that we replace the largest component with the subconfigurations depicted in Fig.~\ref{picture: counting SEQ-irreducible objects}.
For example, for the left configuration in Fig.~\ref{picture: m-sequences}, we consider the configurations in Fig.\ref{picture: (m+1)-sequences} and Fig.~\ref{picture: (m+2)-sequence} taken negative and positive, respectively.
Since
 \[
  \binom{n}{k_1,\ldots,k_m}\b_{k_1}\ldots\b_{k_{m-1}}\a_{k_m} =
  \binom{n}{k_m}\b_{n-k_m}^{(m-1)}\a_{k_m},
 \]
this gives us the structure of $r$ leading terms of the asymptotic expansion of $\b_n^{(m)}$ in the case where the class $\mcA$ is gargantuan: for any positive $r$, as $n\to\infty$,
 \[
  \b_n^{(m)} =
  m\sum\limits_{k=1}^{r-1}
   \Big(\b_k^{(m-1)}-2\b_k^{(m)}+\b_k^{(m+1)}\big)
    \binom{n}{k}\a_{n-k} +
   O(n^r\a_{n-r}).
 \]

\begin{figure}[!ht]
\begin{center}
\begin{tikzpicture}[>= latex, line width=.5pt]
 \begin{scope}[scale=0.9]
  \small
  \def\quot{0.1}  
  \def\R{30pt} 
  \def\r{10pt} 
  \def\h{60pt} 
  \def\d{80pt} 
  \coordinate (c1) at (-2*\h+2*\R-2*\r,0);
  \coordinate (c2) at (-\h+\R-\r,0);
  \coordinate (c3) at (0,0);
  \coordinate (c4) at (\h,0);
  \coordinate (c5) at (\h+\d,0);
  \coordinate (c6) at (2*\h-\R+\r+\d,0);
  \coordinate (c7) at (3*\h-\R+\r+\d,0);
  \coordinate (c8) at (4*\h-\R+\r+\d,0);
  \foreach \p in {c1,c2,c3,c5,c6,c8}{
   \draw (\p) circle (\r);
  }
  \foreach \p in {c4,c7}{
   \draw (\p) circle (\R);
  }
  \foreach \p in {c1,c2,c3,c5,c6} \draw[->] (\p)++(\r,0) -- ++(\h-\R-\r,0);
  \foreach \p in {c7} \draw[->] (\p)++(\R,0) -- ++(\h-\R-\r,0);
 \end{scope}
\end{tikzpicture}
\end{center}
\caption{Specific sequences with three small and one large components.}\label{picture: (m+1)-sequences}
\end{figure}
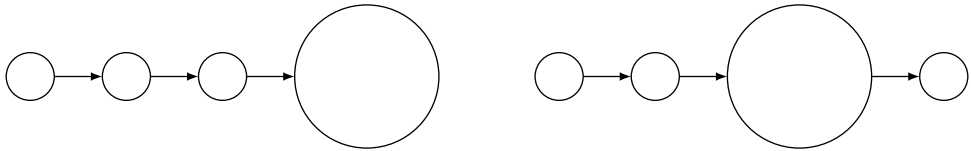

\begin{figure}[!ht]
\begin{center}
\begin{tikzpicture}[>= latex, line width=.5pt]
 \begin{scope}[scale=0.9]
  \small
  \def\quot{0.1}  
  \def\R{30pt} 
  \def\r{10pt} 
  \def\h{60pt} 
  \def\d{80pt} 
  \coordinate (c1) at (-2*\h+2*\R-2*\r,0);
  \coordinate (c2) at (-\h+\R-\r,0);
  \coordinate (c3) at (0,0);
  \coordinate (c4) at (\h,0);
  \coordinate (c5) at (2*\h,0);
  \foreach \p in {c1,c2,c3,c5}{
   \draw (\p) circle (\r);
  }
  \foreach \p in {c4}{
   \draw (\p) circle (\R);
  }
  \foreach \p in {c1,c2,c3} \draw[->] (\p)++(\r,0) -- ++(\h-\R-\r,0);
  \foreach \p in {c4} \draw[->] (\p)++(\R,0) -- ++(\h-\R-\r,0);
 \end{scope}
\end{tikzpicture}
\end{center}
\caption{Specific sequence with four small and one large components.}\label{picture: (m+2)-sequence}
\end{figure}
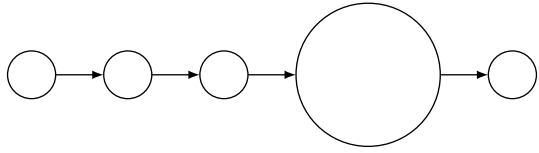

\subsection{Asymptotics of irreducible labeled tournaments}
\label{subsection: tournaments asymptotics}

Recall that a \emph{tournament} is a complete directed graph in which any pair of vertices is joined by exactly one of two possible directed edges.
A tournament is said to be \emph{reducible} if there is a nontrivial partition $A \sqcup B$ of the set of its vertices such that all directed edges between the parts go from $A$ to $B$; otherwise, the tournament is called \emph{irreducible}.
It is known that a~tournament is irreducible if and only if it is strongly connected \cite{Rado1943}, \cite{Roy1958}.
In this section, we assume that the vertices of a tournament of size $n$ are labeled by the elements of the set $[n]$.
 
Let $T$ be a uniform random labeled tournament of size $n$, so that the probability of choosing this particular tournament is $1/2^{\binom{n}{2}}$.
Our goal is to provide the structure of this asymptotic expansion, as well as to establish the asymptotic probability that the tournament $T$ has a given number of irreducible components.

\

Let us introduce the following notations:
\begin{itemize}
 \item 
  $\mcT$ for the labeled combinatorial class of tournaments,
 \item 
  $\mcIT$ for its subclass of irreducible tournaments,
 \item 
  $(\t_n)$ for the counting sequence of the class $\mcT$,
 \item 
  $(\it_n)$ for the counting sequence of the class $\mcIT$,
 \item 
  $(\it_n^{(m)})$ for the counting sequence of the class $\mcIT^m$ (in other words, $\it_n^{(m)}$ is the number of labeled tournaments with $m$ irreducible parts; in particular, $\it_n^{(1)}=\it_n$).
\end{itemize}
Clearly, the total number of labeled tournaments of size $n$ is $\t_n=2^{\binom{n}{2}}$.
There is no close formula for the number of irreducible tournaments of fixed size.
However, it is possible to calculate the values of $(\it_n)$ using the recurrent relation
 \[
  \it_n = \t_n - \sum\limits_{k=1}^{n-1}\binom{n}{k}\it_k\t_{n-k},
 \]
see~\cite{Wright1970jul}.
A slightly different approach that uses generating functions is based on the following structural result.

\begin{lemma}[{\cite[Lemma 1]{MonteilNurligareev2021}}]
\label{lemma: tournament SEQ-decomposition}
 Any tournament can be uniquely decomposed into a sequence of irreducible parts, that is,
 \begin{equation}\label{formula: T=SEQ(IT)}
  \mcT = \SEQ(\mcIT).
 \end{equation}
\end{lemma}
\begin{proof}
 The irreducible parts of a tournament are its strongly connected components $T_i$,
 while the order $\prec$ on these components is determined by reachability:
 $T_i \prec T_j$ if and only if the vertices of $T_i$ are reachable from $T_j$.
\end{proof}

Decomposition~\eqref{formula: T=SEQ(IT)} plays the key role in the combinatorial interpretation of the structure of asymptotics~\eqref{formula: Wright's asymptotic for tournaments}.

\begin{proposition}\label{prop: SEQ_m-asymptotics for tournaments}
 Let $m$ be a fixed positive integer.
 The asymptotic probability that a random labeled tournament $T$ with $n$~vertices consists of $m$ irreducible parts satisfies
 \begin{equation}\label{formula: SEQ_m-asymptotics for tournaments}
  \bbP\big(T\mbox{ has }m\mbox{ irreducible parts}\big)
   \approx
  \sum\limits_{k\ge0} \d_{k,m}(\mcT) \cdot \binom{n}{k} \cdot \dfrac{2^{k(k+1)/2}}{2^{kn}},
 \end{equation}
 where $\d_{k,m}(\mcT) = m\left(\it_k^{(m-1)}-2\it_k^{(m)}+\it_k^{(m+1)}\right)$. In particular, for $m=1$, we have
 \begin{equation}
 \label{formula: SEQ-asymptotics for tournaments}
  \bbP\big(T\mbox{ is irreducible}\big) \approx
  1 - \sum\limits_{k\ge1}
   \left(2\it_k-\it_k^{(2)}\right) \cdot
   \binom{n}{k} \cdot \dfrac{2^{k(k+1)/2}}{2^{kn}}.
 \end{equation}
\end{proposition}
\begin{proof}
 According to Lemma~\ref{lemma: tournament SEQ-decomposition}, we have the structural decomposition $\mcT=\SEQ(\mcIT)$.
 Furthermore, by Lemma~\ref{lemma: sufficient conditions for gargantuan sequence} the sequence $a_n = \t_n/n!$ is gargantuan (for details, see the proof in~\cite[Proposition~5.2]{MonteilNurligareevSET}).
 Thus, relation~\eqref{formula: SEQ_m-asymptotics for tournaments} can be derived from Theorem~\ref{theorem: SEQ_m-asymptotics} applied to the labeled combinatorial class $\mcA = \mcT$.
\end{proof}

\begin{corollary}\label{theorem: leading term for SEQ_m-asymptotics for tournaments}
 Let $m$ be a fixed positive integer.
 The asymptotic probability that a~random labeled tournament $T$ with $n$~vertices consists of $m$ irreducible parts satisfies
 \[
  \bbP\big(T\mbox{ has }m\mbox{ irreducible parts}\big) =
  m \cdot (n)_{m-1} \cdot \dfrac{2^{m(m-1)/2}}{2^{(m-1)n}} +
  O\left(\dfrac{n^{m}}{2^{mn}}\right),
 \]
 where $(n)_{k}=n(n-1)(n-2)\ldots(n-k+1)$ are the falling factorials.
\end{corollary} 
\begin{proof}
 This is a straightforward implication of Proposition~\ref{prop: SEQ_m-asymptotics for tournaments} and Corollary~\ref{corollary: leading term of SEQ_m-asymptotics}.
\end{proof}

\subsection{Irreducible multitournaments and their asymptotics}
\label{subsection: multitournaments asymptotics}

In this section, we study the asymptotic behavior of irreducible multitournaments, which can be considered a natural generalization of the irreducible tournaments discussed above.

Fix a positive integer $d$.
Following our paper~\cite{MonteilNurligareevSET}, we define a \emph{$d$-multitournament} of size $n$ as a directed graph with $n$ vertices such that any pair of its distinct vertices~$i$ and~$j$ is joined by $d$ directed edges.
Here, we assume that the multitournament is labeled, that is, its vertices bear labels from $1$ to $n$.
We also assume that the collinear edges are indistinguishable.
Thus, for each pair of distinct vertices $i$ and $j$, the configuration of the edges between them is determined by a number $\ell\in[d]$ of edges that are directed from $i$ to $j$, while the other $(d-\ell)$ edges are directed from $j$ to $i$.
In particular, the case $d=1$ gives us tournaments discussed in previous sections.

Similarly to tournaments, we say that a $d$-multitournament is \emph{irreducible} if it is strongly connected in the sense of directed graphs.
The labeled class $\mcT(d)$ of $d$-multitour\-naments and its subclass $\mcIT(d)$ of irreducible $d$-multitournaments satisfy the relation
 \begin{equation}\label{formula: T(d)=SEQ(IT(d))}
  \mcT(d) = \SEQ\big(\mcIT(d)\big),
 \end{equation}
which can be justified in the same way as Lemma~\ref{lemma: tournament SEQ-decomposition}.
In addition, from the description of the model it follows that the total number $\t_n(d)$ of $d$-multitournaments of size $n$ is equal to
\[
 \t_n(d) = (d+1)^{\binom{n}{2}}.
\]
These two facts allow us to establish a complete expansion of the asymptotic probability that a random $d$-multitournament is irreducible or consists of a given number of irreducible parts,
which is done in terms of the numbers~$\it_n^{(m)}(d)$ of $d$-multitournaments of size $n$ that have exactly $m$ irreducible parts.

\begin{proposition}\label{prop: SEQ_m-asymptotics for multitournaments}
 Let $d$ and $m$ be two fixed positive integers.
 The asymptotic probability that a random labeled $d$-multitournament $T$ of size $n$ has $m$ irreducible parts satisfies
 \begin{equation}\label{formula: SEQ_m-asymptotics for multitournaments}
  \bbP\big(T\mbox{ has }m\mbox{ irreducible parts}\big) \approx
  \sum\limits_{k\ge0} \d_{k,m}\big(\mcT(d)\big) \cdot \binom{n}{k} \cdot \dfrac{(d+1)^{k(k+1)/2}}{(d+1)^{kn}},
 \end{equation}
 where $\d_{k,m}\big(\mcT(d)\big) = m\left(\it_k^{(m-1)}(d)-2\it_k^{(m)}(d)+\it_k^{(m+1)}(d)\right)$.
 In particular, for $m=1$,
 \[
  \bbP\big(T\mbox{ is irreducible}\big) \approx
  1 - \sum\limits_{k\ge1}
   \left(2\it_k(d)-\it_k^{(2)}(d)\right) \cdot
   \binom{n}{k} \cdot \dfrac{(d+1)^{k(k+1)/2}}{(d+1)^{kn}}.
 \]
\end{proposition} 
\begin{proof}
 To obtain the theorem statement, it is sufficient to apply Theorem~\ref{theorem: SEQ_m-asymptotics} to the labeled combinatorial class $\mcT(d)$.
 This is possible due to relation~\eqref{formula: T(d)=SEQ(IT(d))} and the fact that the sequence $\big(\t_n(d)/n!\big)$ is gargantuan.
 The latter can be proved in the same way as for $d=1$ in the proof of Proposition~\ref{prop: SEQ_m-asymptotics for tournaments} by replacing $2$ with $(d+1)$, see~\cite[Proposition~5.2]{MonteilNurligareevSET} for details.
\end{proof}

\begin{corollary}\label{cor: leading term for SEQ_m-asymptotics for tournaments}
 Let $d$ and $m$ be two fixed positive integers.
 The asymptotic probability that a random labeled $d$-multitournament $T$ of size $n$ has $m$ irreducible parts satisfies
 \[
  \bbP\big(T\mbox{ has }m\mbox{ irreducible parts}\big) =
  m \cdot (n)_{m-1} \cdot \dfrac{(d+1)^{m(m-1)/2}}{(d+1)^{(m-1)n}} +
  O\left(\dfrac{n^{m}}{(d+1)^{mn}}\right),
 \]
 where $(n)_{k}=n(n-1)(n-2)\ldots(n-k+1)$ are the falling factorials.
\end{corollary} 
\begin{proof}
 This is a straightforward implication of Proposition~\ref{prop: SEQ_m-asymptotics for multitournaments} and Corollary~\ref{corollary: leading term of SEQ_m-asymptotics}.
\end{proof}

\begin{remark}
 A slightly different multitournament model is considered in~\cite[Section~10.3.2]{Nurligareev2022}.
 Compared to the model described here, all the edges in a multitournament within that model are distinguishable.
 More precisely, one can consider such a multitournament an element of the Hadamard product $\mcT\odot\ldots\odot\mcT$ of $d$ copies of the class $\mcT$.
 Thus, each pair of its vertices is linked by $d$ directed edges that bear colors from $1$ to $d$, one of each color.
 From a counting point of view, that model is equivalent to $\mcT(2^d-1)$, and there is a bijection that preserves irreducible components.
 Therefore, the asymptotic behavior of that model is also described by Proposition~\ref{prop: SEQ_m-asymptotics for multitournaments}.
\end{remark}

\subsection{Irreducible linear orders and their asymptotics}
\label{subsection: Linear orders asymptotics}

In this section, we study the asymptotic behavior of irreducible multiple linear orders.
In a sense, this can be considered as a preliminary step to indecomposable permutation asymptotics.
Fix a positive integer $d$.
Following~\cite{MonteilNurligareevSET}, we call a \emph{$d$-multiple linear order} of size $n$ a $d$-tuple of linear orders of the set $[n]$.
A $d$-multiple linear order $L=(<_1,\ldots,<_d)$ is \emph{reducible} if there is a nontrivial partition $[n]=A\sqcup B$ such that $a <_k b$ for any pair of elements $(a,b)\in A\times B$ and any linear order $<_k$.
Otherwise, the multiple linear order $L$ is \emph{irreducible}.

We denote by $\mcL(d)$ and $\mcIL(d)$ the (labeled) combinatorial class of $d$-multiple linear orders and its subclass of irreducible $d$-multiple linear orders, respectively.
We also designate by $\big(\l_n(d)\big)$, $\big(\il_n(d)\big)$, and $\big(\il_n^{(m)}(d)\big)$, respectively, the counting sequences of the classes $\mcL(d)$, $\mcIL(d)$, and $\mcIL^m(d)$.
Clearly, the class $\mcL(d)$ can be represented as the Hadamard product
 \[
  \mcL(d) = \mcL\odot\ldots\odot\mcL
 \]
of $d$ copies of the class of linear orders $\mcL=\mcL(1)$.
In particular, $\l_n(d)=(n!)^d$.
Our goal is to establish the asymptotic behavior of the numbers $\il_n(d)$ and, more generally, of the numbers $\il_n^{(m)}(d)$.
For this purpose, it will be useful to decompose the class $\mcL(d)$ into a~sequence of irreducible $d$-multiple linear orders:
 \begin{equation}\label{formula: L(d)=SEQ(IL(d))}
  \mcL(d) = \SEQ\big(\mcIL(d)\big).
 \end{equation}
For a particular element $L = (<_1,\ldots,<_d) \in \mcL_n(d)$, its decomposition into a~sequence of irreducible elements is determined by the partition
 \[
  [n] = I_1 \sqcup \ldots \sqcup I_m
 \]
into the maximum possible number $m$ of intervals such that $a <_k b$ for every pair of elements $(a,b)\in I_i \times I_j$ with $i<j$ and for every linear order $<_k$.
 
\begin{proposition}\label{prop: SEQ_m-asymptotics for linear orders}
 Let $d$ and $m$ be two fixed positive integers.
 The asymptotic probability that a random $d$-multiple linear order $L$ of size $n$ consists of $m$ irreducible parts satisfies
 \begin{equation}\label{formula: SEQ_m-asymptotics for linear orders}
  \bbP\big(L\mbox{ has }m\mbox{ irreducible parts}\big)
   \approx
  \left\{\begin{array}{ll}
   0 & \mbox{if } d=1\\
   \sum\limits_{k\ge0}
    \dfrac{\d_{k,m}\big(\mcL(d)\big)}{k!\cdot\big((n)_k\big)^{d-1}}
     &  \mbox{if } d>1,
  \end{array}
  \right.
 \end{equation}
 where $\d_{k,m}\big(\mcL(d)\big) = m\Big(\il_k^{(m-1)}(d)-2\il_k^{(m)}(d)+\il_k^{(m+1)}(d)\Big)$ and $(n)_k$ are the falling factorials.
 In particular, for $m=1$ and $d>1$,
 \[
  \bbP\big(L\mbox{ is irreducible}\big)
   \approx
  1
   -
  \sum\limits_{k\ge1}
   \dfrac{2\il_k(d)-\il_k^{(2)}(d)}{k!\cdot\big((n)_k\big)^{d-1}},
 \]
\end{proposition}
\begin{proof}
 For $d=1$, a linear order of size $n$ consists of $n$ irreducible parts.
 This fact implies that $\il_n^{(m)}=\One_{n=m}$, and therefore, the probability that $L$ consists of any fixed number~$m$ of irreducible parts tends to $0$, as $n\to\infty$.
 
 For $d\ge2$, relation~\eqref{formula: SEQ_m-asymptotics for linear orders} follows from Theorem~\ref{theorem: SEQ_m-asymptotics} applied to the class~$\mcL(d)$.
 The only fact to verify is the theorem condition that the sequence $a_n = \l_n(d)/n! = (n!)^{d-1}$ is gargantuan.
 For $d = 2$, this can be done using Lemma~\ref{lemma: sufficient conditions for gargantuan sequence}, see~\cite[Proposition 5.10]{MonteilNurligareevSET} for details.
 In the case where $d>2$, it suffices to apply Lemma~\ref{lemma: a_nb_n is gargantuan}.
\end{proof}

\begin{corollary}\label{cor: leading term for SEQ_m-asymptotics for linear orders}
 Let $d$ and $m$ be two fixed positive integers.
 In the case where $d>1$, the asymptotic probability that a random $d$-multiple linear order $L$ of size $n$ consists of $m$~irreducible parts satisfies
 \[
  \bbP\big(L\mbox{ has }m\mbox{ irreducible parts}\big)
   =
  \dfrac{m}{\big((n)_{m-1}\big)^{d-1}}
   +
  O\left(\dfrac{1}{n^{m(d-1)}}\right).
 \]
\end{corollary}
\begin{proof}
 To obtain the leading term of asymptotics~\eqref{formula: SEQ_m-asymptotics for linear orders}, we apply Corollary~\ref{corollary: leading term of SEQ_m-asymptotics}. 
\end{proof}

\section{Lift}
\label{section: Lift}

In this section, we discuss the lift operation and the way it can be applied.
The introduction of this operation is initially motivated by the study of the asymptotic behavior of indecomposable permutations.
The class of indecomposable permutations is not stable under relabeling and, therefore, cannot be treated directly by the tools described in Section~\ref{section: labeled SEQ-asymptotics}.
The lift allows us to transform indecomposable permutations into a truly labeled class, that is, into a class whose structure admits a relabeling, and thus apply Theorem~\ref{theorem: SEQ_m-asymptotics}.

We begin our presentation with recalling some information about indecomposable permutations.
Next, with the help of multiple linear orders discussed in the previous section, we introduce the lift operation.
Finally, we apply lift to obtain the asymptotic probability that a random permutation consists of a given number of indecomposable parts.

\subsection{Indecomposable permutations}
\label{subsection: indecomposable permutations}

Recall that a \emph{permutation} of size $n$ is a~bijection $[n]\to[n]$.
A permutation $\sigma$ is \emph{decomposable} if there is a nontrivial interval $[k]$ invariant under its action, that is,
 \(
  \sigma\big([k]\big) = [k]
 \)
for some $k<n$.
Otherwise, this permutation is called \emph{indecomposable}.
For any permutation, there is the uniaue partition of the set $[n]$ into the maximal possible number $m$ of intervals,
 \[
  [n] = I_1 \sqcup \ldots \sqcup I_m,
 \]
such that $\sigma\big(I_k\big) = I_k$ for any interval $I_k$.
In this case, we say that $\sigma$ consists of $m$ \emph{indecomposable parts}.

For permutations and their counting sequences, we use the following notations:
\begin{itemize}
 \item 
  $\mcP$ for the combinatorial class of all permutations,
 \item 
  $\mcIP$ for the set of indecomposable permutations,
 \item 
  $\mcIP^{(m)}$ for the set of permutations with $m$ indecomposable parts,
 \item 
  $(\p_n)$ for the counting sequence of the class $\mcP$,
 \item 
  $(\ip_n)$ for the counting sequence of the set $\mcIP$,
 \item 
  $\big(\ip_n^{(m)}\big)$ for the counting sequence of the set $\mcIP^{(m)}$.
\end{itemize}
Clearly, we have $\p_n=n!$ and $\ip_n^{(1)}=\ip_n$.

\

Let $\sigma\colon[n]\to[n]$ be a uniform random permutation, so that the probability of choosing this particular permutation is $1/n!$.
The asymptotic expansion of the probability that $\sigma$~is indecomposable was established in 1972 by Comtet~\cite{Comtet1972},~\cite{Comtet1974}.
He showed in particular that the first several terms of this probability are the following:
 \begin{multline*}
  \bbP\big(\sigma\mbox{ is indecomposable}\big) = \\ =
  1 -
  \dfrac{2}{n} -
  \dfrac{1}{(n)_2} -
  \dfrac{4}{(n)_3} -
  \dfrac{19}{(n)_4} -
  \dfrac{110}{(n)_5} -
  \dfrac{745}{(n)_6} -
  \dfrac{5752}{(n)_7} -
  \dfrac{49775}{(n)_8} -
  \dfrac{476994}{(n)_9} -
  \dfrac{5016069}{(n)_{10}} +
  O\left(\dfrac{1}{n^{11}}\right).
 \end{multline*}
Our starting goal is to interpret the numerators in this expansion combinatorially.
After that, for a given positive integer $m$, we proceed to the probability that $\sigma$ possesses exactly $m$ maximum invariant intervals.

\begin{remark}
 We call permutations \emph{indecomposable} following Comtet~\cite[p.~262]{Comtet1974} and Flajolet and Sedgewick~\cite[p.~89]{FlajoletSedgewick2009}.
 However, in the literature, the reader can also find such designations as \emph{irreducible permutations}~\cite{Klazar2003,Ardila2015} or \emph{connected permutations}~\cite{DuchampHivertThibon2002, KohRee2007}.
\end{remark}

\subsection{Lift operation}
\label{subsection: lift}

The goal of this section is to introduce a construction that would allow us to transform the set of indecomposable permutations into a combinatorial class stable under relabeling.
The proposed construction, which we call \emph{lift}, is based on the relations between permutations and linear orders.
From a counting point of view, the classes of permutations $\mcP$ and linear orders $\mcL$ are indistinguishable, just because their counting sequences are the same:
 \[
  \p_n = \l_n = n!.
 \]
Moreover, for any positive integer $n$, there is a natural bijection $\mcL_n \to \mcP_n$:
 \begin{equation}\label{formula: permutation associated with a linear order}
  \ell = (\ell_1 < \ldots < \ell_n)
  \mapsto
  \pi_{\ell} = 
  \left(\begin{array}{ccc}
   1 & \ldots & n \\
   \ell_1 & \ldots & \ell_n
  \end{array}\right)
 \end{equation}
(we call $\pi_{\ell}$ the permutation \emph{associated} with the linear order $\ell$).
However, neither this bijection nor any other bijection can commute with all possible relabelings.
In fact, any two linear orders of the same size have the same structure, and thus one can be obtained from the other by relabeling.
On the other hand, two permutations are equivalent in terms of relabeling if and only if they belong to the same conjugacy class.

At the same time, there is an isomorphism $\Psi \colon \mcL\odot\mcL \to \mcL\odot\mcP$ defined by the relation $\Psi(\ell,\hat{\ell}) = (\ell,\pi_{\hat{\ell}}\pi_{\ell}^{-1})$, where $\pi_{\hat{\ell}}$ is the permutation \emph{associated} with the linear order $\hat{\ell}$.
Explicitly,
 \begin{equation}\label{formula: bijection LL to LP}
  \Psi
  \Big(
   (\ell_1 < \ldots < \ell_n),
   (\hat{\ell}_1 < \ldots < \hat{\ell}_n) 
  \Big)
   =
  \left(
   (\ell_1 < \ldots < \ell_n),
  \left(\begin{array}{ccc}
   \ell_1 & \ldots & \ell_n \\
   \hat{\ell}_1 & \ldots & \hat{\ell}_n
  \end{array}\right)
  \right).
 \end{equation}
The isomorphism $\Psi$ is actually behind the below definition of the lift operation, and it will allow us to make a correspondence between indecomposable permutations and irreducible $2$-linear orders discussed in the previous section.

Before passing to the formal definition, we need to specify how we understand the notion of relabeling.
Given a labeled combinatorial class $\mcA$ and a bijection $\rho\colon[n]\to[n]$, we assume that the \emph{relabeling corresponding to $\rho$} (or the \emph{transport function along $\rho$}) is a~map $R_{\mcA}[\rho] \colon \mcA_n \to \mcA_n$ satisfying $R_{\mcA}[\rho\tau] = R_{\mcA}[\rho] \circ R_{\mcA}[\tau]$ and $R_{\mcA}[\Id_{[n]}] = \Id_{\mcA_n}$.
For instance, for linear orders, we have
 \[
  R_{\mcL}[\rho](\ell) = \rho\ell,
 \]
\emph{i.e.} $\pi_{\ell} \mapsto \rho\pi_{\ell}$ in terms of associated permutations described by relation~\eqref{formula: permutation associated with a linear order}.
In turn, for permutations, relabeling takes the following form:
 \[
  R_{\mcP}[\rho](\sigma) = \rho\sigma\rho^{-1}.
 \]
Particular examples of these two relabelings are shown in their graphical representation in Fig.~\ref{figure: linear order relabeling} and Fig.~\ref{figure: permutation relabeling}.


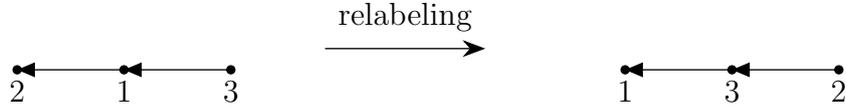
\begin{figure}[!ht]
\begin{center}
\begin{tikzpicture}[line width=.5pt]
 \begin{scope}[>=triangle 45, xshift=8cm]
  \coordinate [label=below:$1$] (a1) at (-40pt,0);
  \coordinate [label=below:$3$] (a2) at (0pt,0);
  \coordinate [label=below:$2$] (a3) at (40pt,0);
  \draw[<-] (a1) -- (a2);
  \draw[<-] (a2) -- (a3);
  \foreach \p in {a1,a2,a3} \filldraw [black] (\p) circle (1.5pt);
 \end{scope}
 \begin{scope}[xshift=3.7cm]
  \draw (0,20pt) node {relabeling};
  \draw[-{Stealth[length=3mm]}] (-30pt,8pt) -- ++(60pt,0);
 \end{scope}
 \begin{scope}[>=triangle 45]
  \coordinate [label=below:$2$] (a1) at (-40pt,0);
  \coordinate [label=below:$1$] (a2) at (0pt,0);
  \coordinate [label=below:$3$] (a3) at (40pt,0);
  \draw[<-] (a1) -- (a2);
  \draw[<-] (a2) -- (a3);
  \foreach \p in {a1,a2,a3} \filldraw [black] (\p) circle (1.5pt);
 \end{scope}
\end{tikzpicture}
\end{center}
\caption{
 Relabeling of
 \(
  \ell = (2 < 1 < 3)
 \)
 corresponding to
 \(
  \rho =
  \left(\begin{array}{ccc}
   1 & 2 & 3 \\
   3 & 1 & 2
  \end{array}\right).
 \)
}\label{figure: linear order relabeling}
\end{figure}

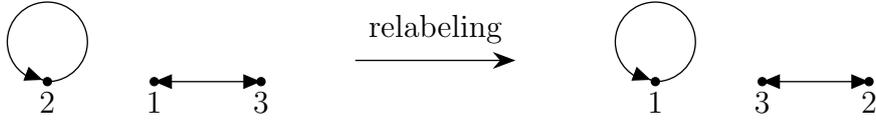
\begin{figure}[!ht]
\begin{center}
\begin{tikzpicture}[line width=.5pt]
 \begin{scope}[>=triangle 45, xshift=8cm]
  \coordinate [label=below:$1$] (a1) at (-40pt,0);
  \coordinate [label=below:$3$] (a2) at (0pt,0);
  \coordinate [label=below:$2$] (a3) at (40pt,0);
  \draw[-] (a1) arc (-90:250:15pt);
  \draw[-<] (a1) -- ++(-8pt,3pt);
  \draw[<->] (a2) -- (a3);
  \foreach \p in {a1,a2,a3} \filldraw [black] (\p) circle (1.5pt);
 \end{scope}
 \begin{scope}[xshift=3.7cm]
  \draw (0,20pt) node {relabeling};
  \draw[-{Stealth[length=3mm]}] (-30pt,8pt) -- ++(60pt,0);
 \end{scope}
 \begin{scope}[>=triangle 45]
  \coordinate [label=below:$2$] (a1) at (-40pt,0);
  \coordinate [label=below:$1$] (a2) at (0pt,0);
  \coordinate [label=below:$3$] (a3) at (40pt,0);
  \draw[-] (a1) arc (-90:250:15pt);
  \draw[-<] (a1) -- ++(-8pt,3pt);
  \draw[<->] (a2) -- (a3);
  \foreach \p in {a1,a2,a3} \filldraw [black] (\p) circle (1.5pt);
 \end{scope}
\end{tikzpicture}
\end{center}
\caption{
 Relabeling of
 \(
  \sigma =
  \left(\begin{array}{ccc}
   1 & 2 & 3 \\
   3 & 2 & 1
  \end{array}\right)
 \)
 corresponding to
 \(
  \rho =
  \left(\begin{array}{ccc}
   1 & 2 & 3 \\
   3 & 1 & 2
  \end{array}\right).
 \)
}\label{figure: permutation relabeling}
\end{figure}

\begin{definition}\label{def: lift}
 Let $\mcB$ be a subset of a labeled combinatorial class $\mcA$.
 We define the class $\Lift(\mcB)$ to be
 \begin{equation}\label{formula: lift}
  \Lift(\mcB) = 
  \Big\{
   \Big. \big(\ell, R_{\mcA}[\pi_{\ell}](b)\big) \,\Big|\,\,
   \ell\in\mcL_n,\,\,
   b\in\mcB_n,\,\,
   n\in\mathbf{Z}_{\ge0}
  \Big\}.
 \end{equation}
\end{definition}

It follows from the definition that $\Lift(\mcB)$ is a labeled combinatorial class, that is, $\Lift(\mcB)$ is stable under relabeling.
From a counting point of view, the class $\Lift(\mcB)$ is similar to the Hadamard product $\mcL\odot\mcB$: they share the same counting sequence $(n!\cdot\b_n)$.
As a consequence, the exponential generating function of $\Lift(\mcB)$ is equal to the ordinary generating function of $\mcB$.
Structurally, $\Lift(\mcB)$ is a subclass of $\mcL\odot\mcA$ that keeps a trace of the properties of the set $\mcB$.

\begin{remark}\label{remark: lift in terms of species}
 We find it more convenient to define the lift operation in terms of species theory (see the book~\cite{BergeronLabelleLeroux1998} for an introduction to the subject).
 Let $F$ be a species of structures, and let $G$ be a set of $F$-structures on positive integers, that is,
 \[
  G \subset \bigcup\limits_{n=0}^{\infty}F[n].
 \]
 We assume that $G$ is not stable under relabeling (since the goal of introducing the lift operation is to fix this issue).
 In other words, there is a bijection $\sigma\colon[n]\to[n]$ such that
 \[
  F[\sigma]\big(G\cap F[n]\big)\neq G\cap F[n].
 \]
 As an example, the reader may think of indecomposable permutations or of trees whose leaves are labeled with odd numbers.
 The purpose of introducing the lift operation is to produce a species of structures that keep track of the set $G$ and some of its properties.
 This species of structures $\Lift(G)$ is defined in the following way.
 \begin{enumerate}
  \item
   For a finite set $U$, the lift operation produces the finite set
   \[
    \Lift(G)[U] = 
    \Big\{
     \big(\ell, F[\pi_u](g)\big) \,\Big|\,\,
     \ell\in\mcL[U],\, g\in G\cap F[n]
    \Big\},
   \]
   where $\pi_u\colon[n]\to U$ is the bijection determined by the linear order $\ell=(u_1,\ldots,u_n)$ on the set $U$, that is, $\pi_u(i)=u_i$.
  \item
   For a bijection $\sigma\colon U\to V$, the lift operation produces the transport function
   \[
    \Lift(G)[\sigma]\colon\Lift(G)[U]\to\Lift(G)[V], \quad
    \Lift(G)[\sigma](\ell,f) = \Big( \mcL[\sigma](\ell), F[\sigma](f) \Big).
   \]
   The transport function is well-defined, since
   \[
    F[\sigma]F[\pi_u] =
    F[\pi_v\pi_u^{-1}]F[\pi_u] = 
    F[\pi_v\pi_u^{-1}\pi_u] = 
    F[\pi_v],
   \]
   while the property $\Lift(G)[\Id_U]=\Id\mid_{\Lift(G)[U]}$ and the transitivity follow from the fact that they hold for the transport functions of the species $\mcL$ and $F$. 
 \end{enumerate}
 Returning to the example where $G$ is the set of trees whose leaves are labeled with odd numbers, we can see six elements of the set $\Lift(G)[U]$ with $U=\{!,?,*\}$ in Fig.~\ref{figure: lift and species}.
\end{remark}

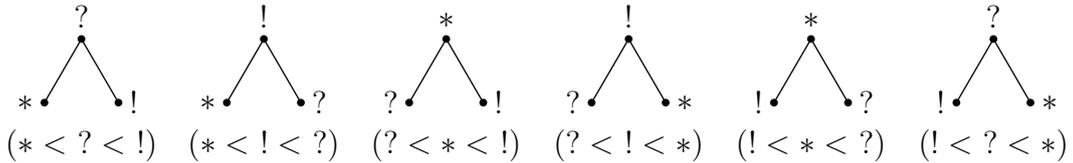
\begin{figure}[!ht]
\begin{center}
\begin{tikzpicture}[scale=.8, line width=.5pt]
 \begin{scope}
  \coordinate [label=90:$?$] (a1) at (0,20pt);
  \coordinate [label=180:$*$] (a2) at ([rotate = 120] a1);
  \coordinate [label=0:$!$] (a3) at ([rotate = 240] a1);
  \draw (a2) -- (a1) -- (a3);
  \draw (0,-30pt) node {$(*<\,\,?<\,\,!)$};
  \foreach \p in {a1,a2,a3} \filldraw [black] (\p) circle (1.5pt);
 \end{scope}
 \begin{scope}[xshift=3cm]
  \coordinate [label=90:$!$] (a1) at (0,20pt);
  \coordinate [label=180:$*$] (a2) at ([rotate = 120] a1);
  \coordinate [label=0:$?$] (a3) at ([rotate = 240] a1);
  \draw (a2) -- (a1) -- (a3);
  \draw (0,-30pt) node {$(*<\,\,!<\,\,?)$};
  \foreach \p in {a1,a2,a3} \filldraw [black] (\p) circle (1.5pt);
 \end{scope}
 \begin{scope}[xshift=6cm]
  \coordinate [label=90:$*$] (a1) at (0,20pt);
  \coordinate [label=180:$?$] (a2) at ([rotate = 120] a1);
  \coordinate [label=0:$!$] (a3) at ([rotate = 240] a1);
  \draw (a2) -- (a1) -- (a3);
  \draw (0,-30pt) node {$(?<*<\,\,!)$};
  \foreach \p in {a1,a2,a3} \filldraw [black] (\p) circle (1.5pt);
 \end{scope}
 \begin{scope}[xshift=9cm]
  \coordinate [label=90:$!$] (a1) at (0,20pt);
  \coordinate [label=180:$?$] (a2) at ([rotate = 120] a1);
  \coordinate [label=0:$*$] (a3) at ([rotate = 240] a1);
  \draw (a2) -- (a1) -- (a3);
  \draw (0,-30pt) node {$(?<\,\,!<*)$};
  \foreach \p in {a1,a2,a3} \filldraw [black] (\p) circle (1.5pt);
 \end{scope}
 \begin{scope}[xshift=12cm]
  \coordinate [label=90:$*$] (a1) at (0,20pt);
  \coordinate [label=180:$!$] (a2) at ([rotate = 120] a1);
  \coordinate [label=0:$?$] (a3) at ([rotate = 240] a1);
  \draw (a2) -- (a1) -- (a3);
  \draw (0,-30pt) node {$(!<*<\,\,?)$};
  \foreach \p in {a1,a2,a3} \filldraw [black] (\p) circle (1.5pt);
 \end{scope}
 \begin{scope}[xshift=15cm]
  \coordinate [label=90:$?$] (a1) at (0,20pt);
  \coordinate [label=180:$!$] (a2) at ([rotate = 120] a1);
  \coordinate [label=0:$*$] (a3) at ([rotate = 240] a1);
  \draw (a2) -- (a1) -- (a3);
  \draw (0,-30pt) node {$(!<\,\,?<*)$};
  \foreach \p in {a1,a2,a3} \filldraw [black] (\p) circle (1.5pt);
 \end{scope}
\end{tikzpicture}    
\end{center}
\caption{Structures $\Lift(G)[U]$ with the set $G$ of trees whose leaves are labeled with odd numbers and the set $U=\{!,?,*\}$.}\label{figure: lift and species}
\end{figure}

\subsection{Asymptotics of indecomposable permutations}
\label{subsection: permutations asymptotics by lift}

In this section, we apply the lift operation to indecomposable permutations and identify the resulting combinatorial class with irreducible $2$-multiple linear orders.
This allows us to obtain the asymptotic probability that a random permutation is irreducible and, more generally, that it consists of a given number of indecomposable parts.

\begin{lemma}\label{lemma: Lift(B) with B in P}
 If $\mcB$ is a subset of the labeled combinatorial class $\mcP$ of permutations, then
 \begin{equation}\label{formula: Lift(B) with B in P}
  \Lift(\mcB) = \Psi
  \left(\Big\{
   \left. (\ell,\hat{\ell}) \in \mcL(2) \,\,\right|\,\,
   \pi_{\ell}^{-1}\pi_{\hat{\ell}} \in \mcB
  \Big\}\right),
 \end{equation}
 where $\pi_{\ell}$ is the permutation associated with the linear order $\ell$ by relation~\eqref{formula: permutation associated with a linear order} and the bijection $\Psi \colon \mcL(2) \to \mcL\odot\mcP$ is defined by relation~\eqref{formula: bijection LL to LP}.
\end{lemma}
\begin{proof}
 According to Definition~\ref{def: lift},
 \[
  \Lift(\mcB) = 
  \Big\{
   \Big. \big(\ell, \pi_{\ell}\sigma\pi_{\ell}^{-1}\big) \,\Big|\,\,
   \ell\in\mcL_n,\,\,
   \sigma\in\mcB_n,\,\,
   n\in\mathbf{Z}_{\ge0}
  \Big\}.
 \]
 On the other hand, for a given permutation $\sigma\in\mcB$, we have
 \[
  \Psi(\ell, \hat{\ell}) = (\ell, \pi_{\ell}\sigma\pi_{\ell}^{-1})
  \qquad\Leftrightarrow\qquad
  \pi_{\hat{\ell}}\pi_{\ell}^{-1} = \pi_{\ell}\sigma\pi_{\ell}^{-1}.
 \]
 The latter relation is equivalent to the fact that $\pi_{\ell}^{-1}\pi_{\hat{\ell}} = \sigma \in \mcB$.
\end{proof}

\begin{lemma}\label{lemma: Lift(IP^m)}
 Let $m$ be a positive integer.
 The set $\mcIP^{(m)}$ of permutations with $m$ indecomposable parts and the class $\mcIL^m(2)$ of $2$-multiple linear orders with $m$ irreducible parts satisfy
 \begin{equation}\label{formula: Lift(IP^m)}
  \Lift\big(\mcIP^{(m)}\big) = \Psi\big(\mcIL^m(2)\big).
 \end{equation}
\end{lemma}
\begin{proof}
 According to the definitions of irreducible $2$-multiple linear orders and indecomposable permutations, there is an equivalence
 \[
  \Big((1<\ldots<n), (i_1<\ldots<i_n)\Big) \in \mcIL^m(2)
  \qquad\Leftrightarrow\qquad
  \left(\begin{array}{ccc}
   1 & \ldots & n \\
   i_1 & \ldots & i_n
  \end{array}\right)
  \in \mcIP^{(m)}.
 \]
 On the other hand, if a $2$-multiple linear order $(\ell, \hat{\ell})$ belongs to the class $\mcIL^m(2)$, then its relabeling $(\rho\ell, \rho\hat{\ell})$ also belongs to this class, and vice versa.
 Taking $\rho=\pi_{\ell}^{-1}$, we get
 \[
  \Big((1<\ldots<n), \pi_{\ell}^{-1}\hat{\ell}\Big) \in \mcIL^m(2)
  \qquad\Leftrightarrow\qquad
  \pi_{\ell}^{-1}\pi_{\hat{\ell}} \in \mcIP^{(m)}.
 \]
 To finish the proof, it suffices to apply Lemma~\ref{lemma: Lift(B) with B in P} to the set $\mcB = \mcIP^{(m)}$.
\end{proof}

\begin{corollary}\label{cor: il_n(2) = n! ip_n}
 For any positive integer $m$, the counting sequences of the set $\mcIP^{(m)}$ and the class $\mcIL^{(m)}(2)$ satisfy
 \[
  \il_n^{(m)}(2) = n!\cdot\ip_n^{(m)}.
 \]
 In particular,
 \[
  \il_n(2) = n!\cdot\ip_n.
 \]
\end{corollary}

\begin{proposition}\label{prop: SEQ_m-asymptotics for permutations}
 Let $m$ be a positive integer.
 The asymptotic probability that a random permutation $\sigma$ of size $n$ consists of $m$ indecomposable parts satisfies
 \begin{equation}\label{formula: SEQ_m-asymptotics for permutations}
  \bbP\big(\sigma\mbox{ has }m\mbox{ indecomposable parts}\big)
   \approx
  \sum\limits_{k\ge0} m \cdot
   \dfrac{\ip_k^{(m-1)}-2\ip_k^{(m)}+\ip_k^{(m+1)}}{(n)_k},
 \end{equation}
 where $(n)_k = n(n-1)\ldots(n-k+1)$ are the falling factorials.
 In particular, for $m=1$,
 \begin{equation}\label{formula: SEQ-asymptotics for permutations}
  \bbP\big(\sigma\mbox{ is indecomposable}\big)
   \approx 1 - \sum\limits_{k\ge1} \dfrac{2\ip_k-\ip_k^{(2)}}{(n)_k},
 \end{equation}
\end{proposition}
\begin{proof}
 Due to Proposition~\ref{prop: SEQ_m-asymptotics for linear orders},
 \[
  \dfrac{\il_n^{(m)}(2)}{\l_n(2)}
   \approx
  \sum\limits_{k\ge0}
   \dfrac{\il_k^{(m-1)}(2) - 2\il_k^{(m)}(2) + \il_k^{(m+1)}(2)}{k! \cdot (n)_k}.
 \]
 Next, according to Corollary~\ref{cor: il_n(2) = n! ip_n}, we have
 \[
  \dfrac{\il_n^{(m)}(2)}{\l_n(2)} = \dfrac{\ip_n^{(m)}}{\p_n}
  \quad\mbox{and}\quad
  \dfrac{\il_k^{(m+1)}(2) - 2\il_k^{(m)}(2) + \il_k^{(m+1)}(2)}{k! \cdot (n)_k} = 
  \dfrac{\ip_k^{(m-1)} - 2\ip_k^{(m)} + \ip_k^{(m+1)}}{(n)_k}.
 \]
 To complete the proof, note that $\ip_n^{(m)}/\p_n$ is the probability that a permutation of size $n$ consists of exactly $m$ indecomposable parts.
\end{proof}

\begin{corollary}\label{cor: leading term for SEQ_m-asymptotics for permutations}
 Let $m$ be a positive integer.
 The asymptotic probability that a random permutation $\sigma$ of size $n$ consists of $m$ indecomposable parts satisfies
 \[
  \bbP\big(\sigma\mbox{ has }m\mbox{ indecomposable parts}\big) =
  \dfrac{m}{(n)_{m-1}} +
  O\left(\dfrac{1}{n^m}\right).
 \]    
\end{corollary}
\begin{proof}
 To obtain the leading term of asymptotic relation~\eqref{formula: SEQ_m-asymptotics for permutations}, we apply Corollary~\ref{cor: leading term for SEQ_m-asymptotics for linear orders} with $d=2$.
\end{proof}

\section{Periodic counting sequences}
\label{section: p-periodic counting sequences}

The asymptotic tools discussed in the previous sections can be employed for labeled combinatorial classes whose counting sequences are gargantuan.
In particular, this condition implies that the number of zero coefficients in any of these sequences is finite.
In some cases, however, we would like to deal with classes whose counting sequences have infinitely many zeros.
For instance, this is the case for the class of perfect matchings.
More generally, we are interested in combinatorial classes whose counting sequences are $p$-periodic for some positive integer $p$, that is,
\begin{enumerate}\label{def: p-periodic sequence}
 \item
  $\a_n\neq0$ for $n=pk$, where $k$ is sufficiently large,
 \item
  $\a_n=0$ in any other case.
\end{enumerate}
We would like to emphasize the reader's attention to the following fact that can be derived from~\cite{Wright1967, Wright1968}.
If a combinatorial class $\mcA$ admits a decomposition $\mcA = \SEQ(B)$ and its counting sequence $(\a_n)$ possesses infinitely many zeros, then $(\a_n)$ is $p$-periodic for some $p>1$.
Thus, in our settings, the scope of possibilities is limited to $p$-periodic sequences.
The most common case in practice is $p=2$.

The purpose of this section is to provide a method for treating these types of cases and to show how it works.
In the first part of our exposition, we adapt Theorem~\ref{theorem: SEQ_m-asymptotics} for combinatorial classes whose counting sequences are $p$-periodic.
In the second part, we apply this adaptation to obtain an asymptotic expansion for the probability that a random perfect matching is indecomposable and, more generally, that it has a given number of indecomposable parts.

\subsection{Asymptotics of classes with \emph{p}-periodic counting sequences}
\label{section: SEQ-asymptotics p-periodic}

Similarly to previous sections, we consider two labeled combinatorial classes $\mcA$ and $\mcB$ that admit a decomposition $\mcA=\SEQ(\mcB)$.
We assume that  their counting sequences are $p$-periodic for some $p\in\mathbb{Z}_{>0}$ and that the sequence $\big(a_{pn}/(pn)!\big)$ is gargantuan.
As usual, for each $m\in\mathbb{Z}_{>0}$, we denote by $\big(\b_{n}^{(m)}\big)$ the counting sequence of the class $\mcB^m$.
Under these conditions, Theorem~\ref{theorem: SEQ_m-asymptotics} can be restated in the following way.

\begin{proposition}\label{prop: SEQ_m-asymptotics, d-gargantuan}
 Let $\mcA$ and $\mcB$ be two labeled combinatorial classes that admit a decomposition $\mcA = \SEQ(\mcB)$.
 Suppose that the counting sequence $(\a_n)$ is $p$-periodic for some positive integer~$p$, and that the sequence $\big(\a_{pn}/(pn)!\big)$ is gargantuan.
 In this case, for any positive integer~$m$, the asymptotic probability that a~random object $a\in\mcA$ of size~$pn$ has $m$ $\SEQ$-irreducible parts satisfies
 \begin{equation}\label{formula: SEQ_m-asymptotics, d-gargantuan}
  \bbP(a\mbox{ has }m\mbox{ }\SEQ\mbox{-irreducible parts}) \approx
  \sum\limits_{k\ge0} \d_{pk,m} \cdot
   \binom{pn}{pk}\cdot \dfrac{\a_{p(n-k)}}{\a_{pn}},
 \end{equation}
 where
 \[
  \d_{pk,m} = 
  m\Big(
   \b_{pk}^{(m-1)} - 2\b_{pk}^{(m)} + \b_{pk}^{(m+1)}
  \Big).
 \]
 In particular, for $m=1$, we have
  \begin{equation}\label{formula: SEQ-asymptotics, d-gargantuan}
  \bbP(a\mbox{ is }\SEQ\mbox{-irreducible}) \approx
  1 - \sum\limits_{k\ge1}
   \Big(2\b_{pk}-\b_{pk}^{(2)}\Big) \cdot
   \binom{pn}{pk}\cdot \dfrac{\a_{p(n-k)}}{\a_{pn}}.
 \end{equation}
\end{proposition}
\begin{proof}
 To establish asymptotic relation~\eqref{formula: SEQ_m-asymptotics, d-gargantuan}, it is sufficient to apply Theorem~\ref{theorem: Bender's} to the formal power series $U(z) = A(z^{1/p})-1$ and the function
 \[
  F(x) = \left(1-\dfrac{1}{1+x}\right)^{m+1}.
 \]
\end{proof}

\begin{corollary}\label{cor: SEQ_m-asymptotics, d-gargantuan}
 If $\a_p\neq0$, then the leading term of asymptotic expansion~\eqref{formula: SEQ_m-asymptotics, d-gargantuan} satisfies
 \begin{multline}\label{formula: leading term of SEQ_m-asymptotics, d-gargantuan}
  \bbP(a\mbox{ has }m\mbox{ irreducible parts}) = \\
 = m \cdot 
   \dfrac{(pn)!}{(p!)^{m-1}(p(n-m+1))!} \cdot
   \dfrac{\a_p^{m-1}\a_{p(n-m+1)}}{\a_{pn}} +
   O\left(n^{pm}\cdot \dfrac{\a_{p(n-m)}}{\a_{pn}}\right).
 \end{multline}
\end{corollary}
\begin{proof}
 The first nonzero coefficient in asymptotics~\eqref{formula: SEQ_m-asymptotics, d-gargantuan} comes from the leading term of the series $B^{m-1}(z)$.
 Since the sequence $(\a_n)$ is $p$-periodic and $\a_p\neq0$, the same is true for the sequence $(\b_n)$, and $\b_p=\a_p$.
 Therefore, the leading term of $B^{m-1}(z)$ is 
 \[
  \b_{p(m-1)}^{(m-1)} \cdot \dfrac{z^{p(m-1)}}{\big(p(m-1)\big)!} = 
  \a_p^{m-1}\cdot\dfrac{z^{p(m-1)}}{(p!)^{m-1}}.
 \]
 Thus, the leading term of~\eqref{formula: SEQ_m-asymptotics, d-gargantuan} is
 \[
  m \cdot \b_{p(m-1)}^{(m-1)} \cdot
  \binom{pn}{p(m-1)} \cdot \dfrac{\a_{p(n-m+1)}}{\a_{pn}}
   =
  m \cdot \dfrac{(pn)!}{(p!)^{m-1}(p(n-m+1))!} \cdot
   \dfrac{\a_p^{m-1}\a_{p(n-m+1)}}{\a_{pn}}.
 \]
\end{proof}

\subsection{Asymptotics of indecomposable perfect matchings}
\label{subsection: perfect matchings asymptotics}

By a \emph{perfect matching} of size $n$, we mean an involution $[n]\to[n]$ without fixed points.
In particular, the class of perfect matchings is a subclass of the class of permutations.
Therefore, perfect matchings admit the concept of indecomposability.

More generally, perfect matchings could be defined in graphs.
In this case, they are understood as unions of graph edges such that each vertex belongs to exactly one of the edges under consideration (see, for example, \cite[p.~374]{LandoZvonkin2004}). 
The perfect matchings studied in this paper correspond to complete labeled graphs.

We denote by
\begin{itemize}
 \item 
  $\mcM$ the labeled combinatorial class of all perfect matchings,
 \item 
  $\mcIM$ its subset of indecomposable perfect matchings,
 \item 
  $\mcIM^{(m)}$ the set of perfect matchings with $m$ indecomposable parts,
 \item 
  $(\m_n)$ the counting sequence of the class $\mcM$,
 \item 
  $(\im_n)$ the counting sequence of the set $\mcIM$,
 \item 
  $\big(\im_n^{(m)}\big)$ the counting sequence of the set $\mcIM^{(m)}$.
\end{itemize}
The counting sequence $(\m_n)$ is $2$-periodic:
 \[
  \m_{2n}=(2n-1)!!,
  \qquad
  \m_{2n+1}=0.
 \]
Our goal is to explore the asymptotic behavior of perfect matchings with respect to their number of indecomposable parts.
To do this, we need the two tools discussed above.
First, as indecomposable perfect matchings are not stable under relabelings, we apply the lift operation.
Second, due to the $2$-periodicity of the sequence $(\m_n)$, we employ Proposition~\ref{prop: SEQ_m-asymptotics, d-gargantuan} instead of Theorem~\ref{theorem: SEQ_m-asymptotics}.

Here is the asymptotic result under consideration.

\begin{proposition}\label{prop: SEQ_m-asymptotics for perfect matchings}
 Let $m$ be a positive integer.
 The asymptotic probability that a random perfect matching $M$ of size $2n$ consists of $m$ indecomposable parts satisfies
 \begin{equation}\label{formula: SEQ_m-asymptotics for perfect matchings}
  \bbP\big(M\mbox{ has }m\mbox{ indecomposable parts}\big) \approx \sum\limits_{k\ge0}\d_{2k,m}(\mcM)\dfrac{\big(2(n-k)-1\big)!!}{(2n-1)!!},
 \end{equation}
 where $\d_{2k,m}(\mcM) = m\Big(\im_{2k}^{(m-1)}-2\im_{2k}^{(m)}+\im_{2k}^{(m+1)}\Big)$. In particular, for $m=1$, we have
 \begin{equation}\label{formula: SEQ-asymptotics for perfect matchings}
  \bbP\big(M\mbox{ is indecomposable}\big)
   \approx
  1 - \sum\limits_{k\ge1} 
   \Big(2\im_{2k}-\im_{2k}^{(2)}\Big)
   \dfrac{\big(2(n-k)-1\big)!!}{(2n-1)!!}.
 \end{equation}
\end{proposition}
\begin{proof}
 First, apply the lift operation to the class $\mcM$.
 According to Lemma~\ref{lemma: Lift(B) with B in P}, we have $\Lift(\mcM) = \Psi\big(\mcLM(2)\big)$, where
 \[
  \mcLM(2) = 
  \Big\{
   \Big. (\ell, \hat{\ell}) \in \mcL(2) \,\Big|\,\,
   \pi_\ell^{-1}\pi_{\hat{\ell}}\in\mcM
  \Big\}
 \]
 is the labeled combinatorial class of \emph{linear matchings} (see Remark~\ref{remark: linear matchings} below), and the map $\Psi\colon\mcL(2)\to\mcL\odot\mcP$ is the isomorphism defined by relation~\eqref{formula: bijection LL to LP}.
 
 Second, observe that the class $\mcLM(2)$ can be decomposed as a sequence:
 \[
  \mcLM(2) = \SEQ\big(\mcILM(2)\big).
 \]
 where the class $\mcILM(2) = \mcLM\cap\mcIL(2)$ consists of \emph{irreducible linear matchings}.
 
 Third, the counting sequence $\big(\lm_n(2)\big)$ of the class $\mcLM(2)$ is $2$-periodic with
 \[
  \lm_{2n}(2) = (2n)!\cdot\m_{2n} = (2n)!\cdot(2n-1)!!,
 \]
 and the sequence $\big(\lm_{2n}(2)/(2n)!\big)$ is gargantuan.
 The latter fact can be established with the help of Lemma~\ref{lemma: sufficient conditions for gargantuan sequence}, see~\cite[Proposition~5.15]{MonteilNurligareevSET} for details.
 
 The second and third imply that we can apply Proposition~\ref{prop: SEQ_m-asymptotics, d-gargantuan} to the class $\mcLM(2)$, which gives us, for a fixed $m\in\bbZ_{>0}$,
 \begin{equation}\label{formula: irreducible linear matchings asymptotics}
  \dfrac{\ilm_{2n}(2)}{\lm_{2n}(2)} \approx
  \sum\limits_{k\ge0} m\Big(\ilm_{2k}^{(m-1)}(2)-2\ilm_{2k}^{(m)}(2)+\ilm_{2k}^{(m+1)}(2)\Big)
   \binom{2n}{2k} \dfrac{\lm_{2(n-k)}(2)}{\lm_{2n}(2)},
 \end{equation}
 where $\big(\ilm_{2n}^{(m)}(2)\big)$ is the counting sequence of the class $\mcILM^m(2)$.
 
 Finally, note that $\Psi\big(\mcILM^m(2)\big) = \Lift(\mcIP^{(m)})$.
 Hence, $\ilm_{2n}^{(m)}(2) = (2n)!\cdot\im_{2n}^{(m)}$, and replacing all occurrences of $\lm_{2n}(2)$ and $\ilm_{2n}^{(m)}(2)$ in asymptotic relation~\eqref{formula: irreducible linear matchings asymptotics} by $(2n)!\cdot(2n-1)!!$ and $(2n)!\cdot\im_{2n}^{(m)}$, respectively, leads to~\eqref{formula: SEQ_m-asymptotics for perfect matchings}. 
\end{proof}

\begin{corollary}\label{corollary: SEQ_m-asymptotics for perfect matchings}
 Let $m$ be a positive integer.
 The asymptotic probability that a random perfect matching $M$ of size $2n$ consists of $m$ indecomposable parts satisfies
 \begin{equation}\label{formula: leading term of SEQ_m-asymptotics for perfect matchings}
  \bbP\big(M\mbox{ has }m\mbox{ indecomposable parts}\big)
   = 
  m\cdot\dfrac{\big(2(n-m)+1\big)!!}{(2n-1)!!} + O\left(\dfrac{1}{n^m}\right).
 \end{equation}
\end{corollary}
\begin{proof}
 There is no perfect matching of size less than $2(m-1)$ that possesses at least $(m-1)$ indecomposable parts.
 Also, there is only one perfect matching of size $2(m-1)$ that consists of $(m-1)$ indecomposable parts: it is $(12)(34)\ldots(2m-3\,2m-2)$ in the cycle notation.
 Therefore, the smallest nonzero coefficient in asymptotics~\eqref{formula: SEQ_m-asymptotics for perfect matchings} is $\d_{2(m-1),m} = m$.
\end{proof}

\begin{remark}\label{remark: linear matchings}
 The class $\mcLM(2)$ admits the following description.
 Following~\cite{MonteilNurligareevSET}, we define a \emph{linear matching} as a~pair of linear orders interchangeable by some relabeling.
 In other words, for any $(\ell,\hat{\ell})\in\mcL(2)$, there is a perfect matching $\rho\in\mcM$ such that
 \[
  (\rho\ell,\rho\hat{\ell}) = (\hat{\ell},\ell).
 \]
 The perfect matching $\rho$ can be expressed as $\rho = \pi_{\hat{\ell}}\pi_{\ell}^{-1} = \pi_{\ell}\pi_{\hat{\ell}}^{-1}$, and hence,
 \[
  \Psi\big(\mcLM(2)\big) =
  \Big\{
   \Big. \big(\ell, \pi_{\hat{\ell}}\pi_{\ell}^{-1}\big) \,\Big|\,\,
   \pi_{\hat{\ell}}\pi_{\ell}^{-1}\in\mcM
  \Big\}.
 \]
 On the other hand, according to Lemma~\ref{lemma: Lift(B) with B in P},
 \[
  \Lift(\mcM) = 
  \Big\{
   \Big. \big(\ell, \pi_{\hat{\ell}}\pi_{\ell}^{-1}\big) \,\Big|\,\,
   \pi_{\ell}^{-1}\pi_{\hat{\ell}}\in\mcM
  \Big\}.
 \]
 Here, we have $\pi_{\ell}^{-1}\pi_{\hat{\ell}} = \pi_{\ell}^{-1}\rho\pi_{\ell}$, and for a fixed $\ell\in\mcL_n$, the conjugation $\rho\mapsto\pi_{\ell}\rho\pi_{\ell}^{-1}$ is a~bijection of the form $\mcM_n\to\mcM_n$, since this map is an automorphism of the group~$\mcP_n$ that preserves the cycle type of its elements.
 Thus, conditions $\pi_{\hat{\ell}}\pi_{\ell}^{-1}\in\mcM$ and $\pi_{\ell}^{-1}\pi_{\hat{\ell}}\in\mcM$ are equivalent, and we have $\Lift(\mcM) = \Psi\big(\mcLM(2)\big)$.
\end{remark}

\section{Asymptotics of the unlabeled class SEQ}
\label{section: unlabeled SEQ-asymptotics}

In this section, we switch from labeled combinatorial classes to unlabeled ones.
The main result presented here is similar to Theorem~\ref{theorem: SEQ_m-asymptotics}.
It is a tool for obtaining complete asymptotic expansions of irreducible combinatorial objects, where irreducibility is understood in terms of the unlabeled construction $\SEQ$.
Using this tool, we revisit the asymptotic behavior of indecomposable permutations and indecomposable perfect matchings, and explore their generalizations.
We also establish the asymptotics of unlabeled tournaments consisting of a given number of irreducible components.

\subsection{Asymptotic probability of unlabeled SEQ-irreducibles}
\label{subsection: unlabeled SEQ-asymptotics}

Recall that we use tilde to emphasize that the combinatorial classes under consideration are unlabeled.
Recall also that, for a positive integer $m$ and an unlabeled combinatorial class $\widetilde{\mcB}$, we denote by $\big(\widetilde{\b}_n^{(m)}\big)$ the counting sequence of the class $\widetilde{\mcB}^m$ of $\widetilde{\mcB}$-object sequences of length $m$.
We extend this notation to the case $m=0$ by adding $\widetilde{\b}_0^{(0)}=1$ and $\widetilde{\b}_n^{(0)}=0$ for any $n>0$ by convention.

\begin{theorem}\label{theorem: SEQ_m-asymptotics, unlabeled}
 Let $\widetilde{\mcA}$ be a gargantuan unlabeled combinatorial class decomposed as a~sequence $\widetilde{\mcA} = \SEQ(\widetilde{\mcB})$ for some unlabeled combinatorial class $\widetilde{\mcB}$.
 Suppose that $a\in{\widetilde{\mcA}}$ is a~random object of size~$n$.
 In this case, for any positive integer $m$,
 \begin{equation}\label{formula: SEQ_m-asymptotics, unlabeled}
  \bbP(a\mbox{ has }m\mbox{ }\SEQ\mbox{-irreducible components}) \approx
  \sum\limits_{k\ge0} m\Big(\widetilde{\b}_k^{(m-1)}-2\widetilde{\b}_k^{(m)}+\widetilde{\b}_k^{(m+1)}\Big) \cdot \dfrac{\widetilde{\a}_{n-k}}{\widetilde{\a}_n}.
 \end{equation}
 In particular, for the case where $m=1$, we have
 \begin{equation}\label{formula: SEQ-asymptotics, unlabeled}
  \bbP(a\mbox{ is }\SEQ\mbox{-irreducible}) \approx
  1 - \sum\limits_{k\ge1}
   \Big(2\widetilde{\b}_k-\widetilde{\b}_k^{(2)}\Big)
    \cdot
   \dfrac{\widetilde{\a}_{n-k}}{\widetilde{\a}_n}.
 \end{equation}
\end{theorem}
\begin{proof}
 Similarly to the proof of Theorem~\ref{theorem: SEQ_m-asymptotics}, for a fixed positive integer $m$ it suffices to apply Theorem~\ref{theorem: Bender's} to the formal power series $U(z) = \widetilde{A}(z) - 1$ and the function
 \[
  F(x) = \left(1-\dfrac{1}{1+x}\right)^{m+1}.
 \]
\end{proof}

\begin{corollary}\label{corollary: leading term of SEQ_m-asymptotics, unlabeled}
 If $\widetilde{\a}_1\neq0$, then the leading term of asymptotic expansion~\eqref{formula: SEQ_m-asymptotics, unlabeled} satisfies
 \[
  \bbP(a\mbox{ has }m\mbox{ }\SEQ\mbox{-irreducible components}) = 
  m \cdot \dfrac{\widetilde{\a}_1^{m-1}\widetilde{\a}_{n-m+1}}{\widetilde{\a}_n} +
  O\left(\dfrac{\widetilde{\a}_{n-m}}{\widetilde{\a}_n}\right).
 \]
\end{corollary}

\subsection{Asymptotics of indecomposable permutations, revisited}
\label{subsection: permutations asymptotics by unlabeled theorem}

Here, we revisit the asymptotic probabilities of indecomposable permutations and indecomposable perfect matchings in view of unlabeled combinatorial classes.
The key idea is to formally treat permutations and perfect matchings as unlabeled objects, ignoring the labels involved in their definitions.
With this approach, indecomposable permutations become a well-defined subclass $\widetilde{\mcIP}$ of the unlabeled combinatorial class~$\widetilde{\mcP}$ of permutations that satisfies
\begin{equation}\label{formula: P=SEQ(IP), unlabeled}
 \widetilde{\mcP} = \SEQ(\widetilde{\mcIP}),
\end{equation}
see~\cite[Example~I.19]{FlajoletSedgewick2009}.
To identify indecomposable elements of a particular permutation $\sigma \in \widetilde{\mcP}_n$, it is sufficient to represent the set $[n]$ as a disjoint union of the maximum possible number $m$ of intervals, $[n] = I_1 \sqcup \ldots \sqcup I_m$, such that each interval $I_k$ is stable under the action of $\sigma$, that is, $\sigma(I_k) = I_k$, see Fig.~\ref{figure: permutation SEQ-decomposition} for an example.

\begin{figure}[!ht]
\begin{center}
\begin{tikzpicture}[>= latex, line width=.5pt]
 \def\quot{0.1}  
 \def\h{30pt}
 \def\hh{\h/1.41}
 \begin{scope}
  \coordinate[label=-90:$1$] (c1) at (0,0);
  \coordinate[label=-90:$2$] (c2) at (\h,0);
  \coordinate[label=-90:$3$] (c3) at (2*\h,0);
  \coordinate[label=-90:$4$] (c4) at (3*\h,0);
  \coordinate[label=-90:$5$] (c5) at (4*\h,0);
  \coordinate[label=-90:$6$] (c6) at (5*\h,0);
  \draw[->] (c1) arc (180:5:\h);
  \draw[->] (c2) arc (45:125:\hh);
  \draw[->] (c3) arc (45:125:\hh);
  \draw[->] (c4) arc (270:-75:\h/4);
  \draw[->] (c5) arc (180:5:\h/2);
  \draw[->] (c6) arc (45:125:\hh);
  \draw[red] (2.5*\h,\h) -- ++(0,-1.5*\h);
  \draw[red] (3.5*\h,\h) -- ++(0,-1.5*\h);
  \foreach \p in {c1,c2,c3,c4,c5,c6}
   \filldraw (\p) circle (1.2pt);
 \end{scope}
 \begin{scope}[xshift=5cm]
  \draw (1.5*\h,0) node {$\leadsto$};
 \end{scope}
 \begin{scope}[xshift=8cm]
  \coordinate[label=-90:$1$] (c1) at (0,0);
  \coordinate[label=-90:$2$] (c2) at (\h,0);
  \coordinate[label=-90:$3$] (c3) at (2*\h,0);
  \draw (3*\h,0) node {$\times$};
  \coordinate[label=-90:$1$] (c4) at (4*\h,0);
  \draw (5*\h,0) node {$\times$};
  \coordinate[label=-90:$1$] (c5) at (6*\h,0);
  \coordinate[label=-90:$2$] (c6) at (7*\h,0);
  \draw[->] (c1) arc (180:5:\h);
  \draw[->] (c2) arc (45:125:\hh);
  \draw[->] (c3) arc (45:125:\hh);
  \draw[->] (c4) arc (270:-75:\h/4);
  \draw[->] (c5) arc (180:5:\h/2);
  \draw[->] (c6) arc (45:125:\hh);
  \foreach \p in {c1,c2,c3,c4,c5,c6}
   \filldraw (\p) circle (1.2pt);
 \end{scope}
\end{tikzpicture}
\end{center}
 \caption{Decomposition of the permutation $\left(\begin{array}{cccccc}
  1 & 2 & 3 & 4 & 5 & 6 \\
  3 & 1 & 2 & 4 & 6 & 5
 \end{array}\right)$ into a sequence.}
 \label{figure: permutation SEQ-decomposition}
\end{figure}
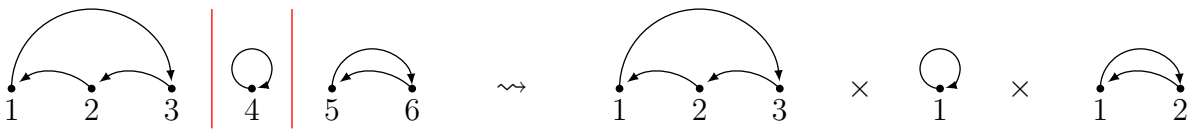

Taking into account what is discussed in the above paragraph, we can prove Proposition~\ref{prop: SEQ_m-asymptotics for permutations} in the following way.

\begin{proof}[Proof of Proposition~\ref{prop: SEQ_m-asymptotics for permutations}]
 As we have just mentioned, $\widetilde{\mcP} = \SEQ(\widetilde{\mcIP})$.
 Also, the unlabeled combinatorial class $\widetilde{\mcP}$ is gargantuan, since, as we have seen in the proof of Proposition~\ref{prop: SEQ_m-asymptotics for linear orders}, its counting sequence $\p_n=n!$ is gargantuan.
 Therefore, we can apply Theorem~\ref{theorem: SEQ_m-asymptotics, unlabeled} to $\widetilde{\mcA} = \widetilde{\mcP}$, which directly implies relation~\eqref{formula: SEQ_m-asymptotics for permutations}.
\end{proof}

Essentially, the case of indecomposable perfect matchings is treated in the same manner.
However, there is a nuance that we would like to emphasize.
Compared to the labeled combinatorial objects whose sizes coincide with the cardinality of the ground set of labels, there is no intrinsic linkage for the unlabeled ones.
\emph{De facto}, in the unlabeled case, the sizes of the objects, as long as they are integers, are defined up to multiplication on a constant.
In particular, it is convenient to count perfect matchings with respect to the number of linked pairs, not to the number of atoms.
Thus, we assume that the counting sequence of the unlabeled combinatorial class $\widetilde{\mcM}$ of perfect matchings is $(\m_{2n})$.
Since this sequence is gargantuan, and since $\widetilde{\mcM} = \SEQ(\widetilde{\mcIM})$, where $\widetilde{\mcIM} \subset \widetilde{\mcM}$ is the subclass of indecomposable perfect matchings, Theorem~\ref{theorem: SEQ_m-asymptotics, unlabeled} is applicable.
As a result, we obtain asymptotic relation~\eqref{formula: SEQ_m-asymptotics for perfect matchings} for the probability that a random perfect matching consists of a given number of indecomposable parts.
Thus, we get another proof of Proposition~\ref{prop: SEQ_m-asymptotics for perfect matchings}.

\subsection{Asymptotics of indecomposable multipermutations}
\label{subsection: multipermutations asymptotics}

In the spirit of $d$-multitournaments and $d$-multiple linear orders, for every positive integer~$d$ we can define the class of \emph{$d$-multipermutations} and its subclass of \emph{$d$-multiple perfect matchings} as the Hadamard products of $d$ copies of $\mcP$ and $\mcM$, respectively:
 \[
  \mcP(d) = \mcP\odot\ldots\odot\mcP
  \qquad\mbox{and}\qquad
  \mcM(d) = \mcM\odot\ldots\odot\mcM.
 \]
We say that an object $\Sigma = (\sigma_1,\ldots,\sigma_d) \in \mcP_n(d)$ is \emph{decomposable} if there is a positive integer $k<n$ such that the interval $[k]$ is invariant under the action of all elements of~$\Sigma$, that is, $\sigma_j\big([k]\big) = [k]$ for all $j\in[d]$.
Otherwise, we call $\Sigma$ \emph{indecomposable}.

In the case where the classes of $d$-multipermutations and $d$-multiple perfect matchings are understood as unlabeled, there are natural decompositions into the sequences of the corresponding indecomposable structures:
 \[
  \widetilde{\mcP(d)} = \SEQ\big(\widetilde{\mcIP}(d)\big)
  \qquad\mbox{and}\qquad
  \widetilde{\mcM}(d) = \SEQ\big(\widetilde{\mcIM}(d)\big).
 \]
In addition, the counting sequences
 \[
  \p_n(d) = (n!)^d
  \qquad\mbox{and}\qquad
  \m_n(d) = \big((2n-1)!!)^d
 \]
of the classes $\widetilde{\mcP}(d)$ and $\widetilde{\mcM}(d)$, respectively, are gargantuan due to Lemma~\ref{lemma: a_nb_n is gargantuan}.
These two facts allow us to apply Theorem~\ref{theorem: SEQ_m-asymptotics, unlabeled}, which gives, for a given positive integer $m$, the asymptotic expansion of the probability that a~random object $\Sigma\in{\widetilde{\mcP}(d)}$ has exactly $m$~indecomposable parts,
 \begin{equation}\label{formula: SEQ_m-asymptotics for multipermutations}
  \bbP\big(\Sigma\mbox{ has }m\mbox{ indecomposable parts}\big)
   \approx
  \sum\limits_{k\ge0} \dfrac{\d_{k,m}\big(\widetilde{\mcP}(d)\big)}{\big((n)_k\big)^{d}},
 \end{equation}
as well as that of a random object $\Sigma'\in{\widetilde{\mcM}(d)}$,
 \begin{equation}\label{formula: SEQ_m-asymptotics for multiple perfect matchings}
  \bbP\big(\Sigma'\mbox{ has }m\mbox{ indecomposable parts}\big)
   \approx
  \sum\limits_{k\ge0} \d_{2k,m}\big(\widetilde{\mcM}(d)\big)
   \left(\dfrac{\big(2(n-k)-1\big)!!}{(2n-1)!!}\right)^d,
 \end{equation}
where 
 \[
  \d_{k,m}\big(\widetilde{\mcP}(d)\big)
  =
  m\Big(\ip_{k}^{(m-1)}(d)-2\ip_{k}^{(m)}(d)+\ip_{k}^{(m+1)}(d)\Big)
 \]
and
 \[
  \d_{2k,m}\big(\widetilde{\mcM}(d)\big)
  =
  m\Big(\im_{2k}^{(m-1)}(d)-2\im_{2k}^{(m)}(d)+\im_{2k}^{(m+1)}(d)\Big)
 \]
Here, we designate by $\big(\ip_n^{(m)}(d)\big)$ and $\big(\im_{2n}^{(m)}(d)\big)$ the counting sequences of the unlabeled combinatorial classes $\widetilde{\mcP}^m(d)$ and $\widetilde{\mcM}^m(d)$, respectively.

Asymptotic relations~\eqref{formula: SEQ_m-asymptotics for multipermutations} and~\eqref{formula: SEQ_m-asymptotics for multiple perfect matchings} can also be established within the labeled case, using the lift operation and Theorem~\ref{theorem: SEQ_m-asymptotics}.
We will not provide exhaustive details here, since the technique is essentially the same as for permutations and perfect matchings in Sections~\ref{subsection: permutations asymptotics by lift} and~\ref{subsection: perfect matchings asymptotics}, respectively.
We only mention that
 \[
  \Lift\big(\mcIP(d)\big) = \Psi\big(\mcIL(d+1)\big),
 \]
while the class $\Psi\big(\Lift\big(\mcIM(d)\big)\big)^{-1} \subset \mcIL(d+1)$ consists of irreducible $(d+1)$-multiple linear orders $(\ell_0,\ldots,\ell_d)$ such that for every $j\in[d]$ there exists a permutation $\rho_j$ satisfying $\rho_j\ell_0=\ell_j$ and $\rho_j\ell_j=\ell_0$.

\begin{remark}
 In our investigation, we introduced a multipermutation as a tuple of permutations of the same size.
 Thus, this notion is different from the one of Comtet who defined a $d$-permutation of $[n]$, in matrix notation, as a relation where all vertical sections and all horizontal sections have $d$ elements, see~\cite[p.~235]{Comtet1974}.
 This notion is also different from the concept of multiset permutations that are often called multipermutations in the literature, see, for example, \cite[p.~127]{Stanley2012}.
\end{remark}

\subsection{Asymptotics of irreducible unlabeled tournaments}
\label{subsection: tournaments asymptotics, unlabeled}

In this section, we focus our attention on the (unlabeled) class $\widetilde{\mcT}$ of unlabeled tournaments.
Compared to the labeled tournaments introduced in Section~\ref{subsection: tournaments asymptotics}, the vertices of the unlabeled tournaments do not bear any labels.
In other words, unlabeled tournaments are counted up to isomorphism.
On the other hand, the notion of irreducibility is left unchanged, and we still have the decomposition
\begin{equation}\label{formula: T=SEQ(IT), unlabeled}
 \widetilde{\mcT} = \SEQ(\widetilde{\mcIT}),
\end{equation}
where $\widetilde{\mcIT}$ designates the class of irreducible unlabeled tournaments.

Let $\widetilde{\t}_n$ and $\widetilde{\it}_n$ be, respectively, the number of unlabeled tournaments of size $n$ and the number of those of them that are irreducible.
Wright provided a method to obtain the asymptotics of $\widetilde{\it}_n$ in terms of $\widetilde{\t}_n$.
In particular, he showed that
 \[
  \widetilde{\it}_n =
  \widetilde{\t}_n -
  2\widetilde{\t}_{n-1} +
  \widetilde{\t}_{n-2} -
  2\widetilde{\t}_{n-3} -
  10\widetilde{\t}_{n-5} +
  O(\widetilde{\t}_{n-6}),
 \]
see \cite[formula (8)]{Wright1970jul}.
Our goal is to establish the complete asymptotic expansion of the number $\widetilde{\it}_n^{(m)}$ of unlabeled tournaments of size $n$ that consists of a given number $m$ of irreducible components.
In addition to decomposition~\eqref{formula: T=SEQ(IT), unlabeled}, the key ingredient to reaching this goal is the asymptotic behavior of the sequence $(\widetilde{t}_n)$:
\begin{equation}\label{formula: tilde(t_n)=t_n/n!}
 \widetilde{\t}_n
  =
 \dfrac{\t_n}{n!} + O\left(\dfrac{\t_{n-2}}{(n-3)!}\right)
  =
 \dfrac{2^{n(n-1)/2}}{n!} + O\left(\dfrac{2^{(n-2)(n-3)/2}}{(n-3)!}\right).
\end{equation}
Here, as in Section~\ref{section: labeled SEQ-asymptotics}, $\t_n$ stands for the number of labeled tournaments of size $n$.
Relation~\eqref{formula: tilde(t_n)=t_n/n!} can be obtained from the exact formula for $\widetilde{\t}_n$, see~\cite{Moon1968}, and means that the majority of unlabeled tournaments of large size $n$ do not have any nontrivial symmetry.

\begin{proposition}\label{prop: SEQ_m-asymptotics for tournaments, unlabeled}
 Let $m$ be a fixed positive integer.
 The asymptotic probability that a~random unlabeled tournament $T$ of size $n$ consists of $m$ irreducible parts satisfies
 \begin{equation}\label{formula: SEQ_m-asymptotics for tournaments, unlabeled}
  \bbP\big(T\mbox{ has }m\mbox{ irreducible parts}\big)
   \approx
  \sum\limits_{k\ge0}
   m\left(
    \widetilde{\it}_k^{(m-1)}
     -
    2\widetilde{\it}_k^{(m)}
     +
    \widetilde{\it}_k^{(m+1)}
   \right)
    \cdot
   \dfrac{\widetilde{\t}_{n-k}}{\widetilde{\t}_n}.
 \end{equation}
 In particular, for $m=1$, we have
 \begin{equation}
 \label{formula: SEQ-asymptotics for tournaments, unlabeled}
  \bbP\big(T\mbox{ is irreducible}\big) \approx
  1 - \sum\limits_{k\ge1}
   \left(2\widetilde{\it}_k-\widetilde{\it}_k^{(2)}\right) \cdot
   \binom{n}{k} \cdot \dfrac{\widetilde{\t}_{n-k}}{\widetilde{\t}_n}.
 \end{equation}
\end{proposition}
\begin{proof}
 As we have already mentioned, similarly to the labeled case, unlabeled tournaments satisfy structural decomposition~\eqref{formula: T=SEQ(IT), unlabeled}.
 Furthermore, the counting sequence $(\widetilde{\t}_n)$ is gargantuan.
 The latter follows from relation~\eqref{formula: tilde(t_n)=t_n/n!}, Lemma~\ref{lemma: Ka_n+b_n is gargantuan} and the fact that the sequence $(\t_n)$ is gargantuan.
 Therefore, we can apply Theorem~\ref{theorem: SEQ_m-asymptotics, unlabeled} to the unlabeled combinatorial class $\widetilde{\mcT}$, which leads directly to asymptotic relation~\eqref{formula: SEQ_m-asymptotics for tournaments, unlabeled}.
\end{proof}

\begin{corollary}\label{theorem: leading term for SEQ_m-asymptotics for tournaments, unlabeled}
 Let $m$ be a fixed positive integer.
 The asymptotic probability that a~random unlabeled tournament $T$ of size $n$ consists of $m$ irreducible parts satisfies
 \[
  \bbP\big(T\mbox{ has }m\mbox{ irreducible parts}\big) =
  m \cdot (n)_{m-1} \cdot \dfrac{2^{m(m-1)/2}}{2^{(m-1)n}} +
  O\left(\dfrac{n^{m}}{2^{mn}}\right),
 \]
 where $(n)_{k}=n(n-1)(n-2)\ldots(n-k+1)$ are the falling factorials.
\end{corollary} 
\begin{proof}
 This follows immediately from Proposition~\ref{prop: SEQ_m-asymptotics for tournaments, unlabeled}, Corollary~\ref{corollary: leading term of SEQ_m-asymptotics, unlabeled} and asymptotic relation~\eqref{formula: tilde(t_n)=t_n/n!}.
\end{proof}

\section{Conclusion}
\label{section: Conclusion}

We have seen the method that allowed us to obtain the complete asymptotic expansion of the probability that a random object from a gargantuan combinatorial class admitting a~$\SEQ$ decomposition is $\SEQ$-irreducible, as well as the probability that this object consists of a given number of $\SEQ$-irreducible components, both in the labeled and unlabeled cases.
It would be natural to wonder if this method can be extended to other types of decomposition.
Here we briefly discuss the results for the cases where the answer is positive, as well as the obstacles arising for other decompositions.

The first of the results concerns the gargantuan labeled combinatorial classes that admit $\CYC$ decomposition, that is, $\mcA = \CYC(\mcB)$.
For a fixed positive integer $m$, we apply Theorem~\ref{theorem: Bender's} with $U(z) = A(z)-1$ and $F(x) = (1-e^{-x})^m/m$ to obtain the following theorem.

\begin{theorem}[$CYC$-asymptotics]\label{theorem: CYC_m-asymptotics}
 Let $\mcA$ be a gargantuan labeled combinatorial class satisfying $\mcA = \CYC(\mcB)$ for some labeled combinatorial class $\mcB$.
 Suppose that $a\in\mcA$ is a~random object of size~$n$.
 In this case, for any positive integer $m$,
 \begin{equation}\label{formula: CYC_m-asymptotics}
  \bbP(a\mbox{ has }m\mbox{ }\CYC\mbox{-irreducible components}) \approx
  \sum\limits_{k\ge0} \Big(\b_k^{(m-1)}-\b_k^{(m)}\Big) \cdot
   \binom{n}{k} \cdot \dfrac{\a_{n-k}}{\a_n},
 \end{equation}
 where $\big(\b_k^{(m)}\big)$ is the counting sequence of the class $\mcB^m$.
 In particular, for $m=1$, we have
 \begin{equation}\label{formula: CYC-asymptotics}
  \bbP(a\mbox{ is }\CYC\mbox{-irreducible}) \approx
  1 - \sum\limits_{k\ge1}
   \b_k \cdot \binom{n}{k} \cdot \dfrac{\a_{n-k}}{\a_n}.
 \end{equation}
\end{theorem}

Another result concerns sets.
In our previous paper~\cite{MonteilNurligareevSET}, we have seen that our method works for gargantuan labeled combinatorial classes that admit a double $\SET/\SEQ$ decomposition.
It turns out that the unlabeled case can be treated in the same manner.

\begin{theorem}\label{theorem: MSET-asymptotics}
 Let $\mcA$ be a gargantuan unlabeled combinatorial class satisfying
 \[
  \mcA = \MSET(\mcC) = \SEQ(D)
  \qquad\mbox{or}\qquad
  \mcA = \PSET(\mcC) = \SEQ(D)
 \]
 for some unlabeled combinatorial classes $\mcC$ and $\mcD$.
 Suppose that $a\in\mcA$ is a~random object of size~$n$.
 In this case,
 \begin{equation}\label{formula: MSET-asymptotics}
  \bbP(a\mbox{ is }\SET\mbox{-irreducible}) \approx
  1 - \sum\limits_{k\ge1} \d_k \cdot \dfrac{\a_{n-k}}{\a_n}.
 \end{equation}
\end{theorem}

Note that the proof of Theorem~\ref{theorem: MSET-asymptotics} demands a more general version of Bender's theorem than the one we used throughout this paper.
The main idea of its application can be found in~\cite[Corollary~6]{Bender1975}, see also \cite[Theorem~7.5.1]{Nurligareev2022} for details.

On the other hand, in both cases, labeled and unlabeled, the existence of $\SEQ$ decomposition is necessary to obtain the combinatorial interpretation of the coefficients involved in the asymptotic expansion of $\SET$-irreducibles.
If there is no explicit $\SEQ$ decomposition, it is still possible to establish a complete asymptotic expansion of $\SET$-irreducibles using Bender's theorem.
However, to give the coefficients a combinatorial meaning, we would like to construct an ``anti-$\SEQ$'' operator.
In the labeled case, we will do that in the forthcoming paper~\cite{MonteilNurligareevANTISEQ}.
We will also interpret the coefficients in the asymptotic probability that a random labeled object consists of a given number of $\SET$-irreducible (aka connected) components.

At the same time, for the unlabeled construction $\CYC$, as well as for the restricted unlabeled constructions $\CYC_m$, $\MSET_m$ and $\PSET_m$, no method providing complete asymptotic expansions is known.
This can be explained by complicated forms of the corresponding ordinary generating functions~\eqref{formula: GF-CYC_m} and~\eqref{formula: GF-MSET_m}.
Compared to unlabeled constructions $\SEQ$ and $\SEQ_m$, the simplified version of Bender's theorem (Theorem~\ref{theorem: Bender's}) cannot be applied here, while the applicability of the general version is still questionable.

The last question that we would like to discuss here concerns the positivity of asymptotic coefficients.
We have seen that, under the conditions of Theorem~\ref{theorem: MSET-asymptotics}, as well as of its labeled analogue~\cite[Theorem~4.1]{MonteilNurligareevSET}, the considered coefficients are nonnegative, since they count objects of certain combinatorial classes.
A similar situation is observed for Theorem~\ref{theorem: CYC_m-asymptotics} in case $m=1$.
On the other hand, for the constructions $\SEQ_m$ and $\CYC_m$, this is no longer the case.
Theorems~\ref{theorem: SEQ_m-asymptotics}, \ref{theorem: SEQ_m-asymptotics, unlabeled} and~\ref{theorem: CYC_m-asymptotics} show that the coefficients are now linear combinations of certain counting sequences, which means that, in principle, they could be negative.
Moreover, the example of irreducible tournaments shows that this is indeed the case (see Proposition~\ref{prop: SEQ_m-asymptotics for tournaments} and Table~\ref{table: d_(k,m)(T)} for the labeled case, as well as Proposition~\ref{prop: SEQ_m-asymptotics for tournaments, unlabeled} and Table~\ref{table: d_(k,m)(tilde(T))} for the unlabeled case).
Thus, there is no hope to interpret them as a counting sequence of any combinatorial class.
One possible research direction is to establish conditions for these coefficients to be nonnegative.
Another possibility consists in changing the paradigm, the point of view.
Following the second approach, in our next paper~\cite{MonteilNurligareevANTISEQ} we interpret the coefficients under consideration in terms of species theory, using the concept of virtual species.

\section*{Acknowledgements.}
Khaydar Nurligareev was supported by the project PICS ANR-22-CE48-0002, as well as the ANR-FWF project PAnDAG ANR-23-CE48-0014-01, both funded by the Agence Nationale de la Recherche.

\bibliographystyle{abbrv}
\bibliography{bibliography}

@Article{BellBenderCameronRichmond2000,
 Title={Asymptotics for the probability of connectedness and the distribution of number of components},
 Author={Bell, Jason P and Bender, Edward A and Cameron, Peter J and Richmond, L Bruce},
 FJournal = {{The Electronic Journal of Combinatorics}},
 Journal = {{Electron. J. Comb.}},
 Volume={7},
 Pages={R33--R33},
 Year={2000}
}

@Article{Bender1974,
  title={Asymptotic methods in enumeration},
  author={Bender, Edward A},
  journal={SIAM review},
  volume={16},
  number={4},
  pages={485--515},
  year={1974},
  publisher={SIAM}
}

@Article{Bender1975,
 Author = {Edward A. {Bender}},
 Title = {{An asymptotic expansion for the coefficients of some formal power series.}},
 FJournal = {{Journal of the London Mathematical Society. Second Series}},
 Journal = {{J. Lond. Math. Soc., II. Ser.}},
 ISSN = {0024-6107; 1469-7750/e},
 Volume = {9},
 Pages = {451--458},
 Year = {1975},
 Publisher = {John Wiley \& Sons, Chichester; London Mathematical Society, London},
 Language = {English},
 MSC2010 = {05A15 05C20 41A10},
 Zbl = {0297.05013}
}

@Book{BergeronLabelleLeroux1998,
 Author = {F. {Bergeron} and G. {Labelle} and P. {Leroux}},
 Title = {{Combinatorial species and tree-like structures. Transl. from the French by Margaret Readdy}},
 FJournal = {{Encyclopedia of Mathematics and Its Applications}},
 Journal = {{Encycl. Math. Appl.}},
 ISSN = {0953-4806},
 Volume = {67},
 ISBN = {0-521-57323-8},
 Pages = {xx + 457},
 Year = {1998},
 Publisher = {Cambridge: Cambridge University Press},
 Language = {English},
 MSC2010 = {05-02 05E99},
 Zbl = {0888.05001}
}

@Book{Bona2015,
 Editor = {B{\'o}na, Mikl{\'o}s},
 Title = {Handbook of enumerative combinatorics},
 FSeries = {Discrete Mathematics and its Applications},
 Series = {Discrete Math. Appl. (Boca Raton)},
 ISBN = {978-1-4822-2085-8; 978-1-4822-2086-5},
 Year = {2015},
 Publisher = {Boca Raton, FL: CRC Press},
 Language = {English},
 DOI = {10.1201/b18255},
 zbMATH = {6398608},
 Zbl = {1314.05001}
}

@InCollection{Ardila2015,
 Author = {Ardila, Federico},
 Title = {Algebraic and geometric methods in enumerative combinatorics},
 BookTitle = {Handbook of enumerative combinatorics},
 ISBN = {978-1-4822-2085-8; 978-1-4822-2086-5},
 Pages = {3--172},
 Year = {2015},
 Publisher = {Boca Raton, FL: CRC Press},
 Language = {English},
 zbMATH = {6490936},
 Zbl = {1330.05009}
}

@Article{Comtet1972,
 Author = {Louis {Comtet}},
 Title = {{Sur les coefficients de l'inverse de la s\'erie formelle \(\sum n! t^n\).}},
 FJournal = {{Comptes Rendus Hebdomadaires des S\'eances de l'Acad\'emie des Sciences, S\'erie A}},
 Journal = {{C. R. Acad. Sci., Paris, S\'er. A}},
 ISSN = {0366-6034; 0302-8429},
 Volume = {275},
 Pages = {569--572},
 Year = {1972},
 Publisher = {Gauthier-Villars, Paris},
 Language = {French},
 MSC2010 = {05A10 05A19},
 Zbl = {0246.05003}
}

@Misc{Comtet1974,
 Author = {Comtet, Louis},
 Title = {Advanced combinatorics. {The} art of finite and infinite expansions. {Translated} from the {French} by {J}. {W}. {Nienhuys}. {Rev}. and enlarged ed},
 Year = {1974},
 Language = {English},
 HowPublished = {Dordrecht, {Holland} - {Boston}, {U}.{S}.{A}.: {D}. {Reidel} {Publishing} {Company}. {X}, 343 p. {Dfl}. 65.00 (1974).},
 zbMATH = {3443655},
 Zbl = {0283.05001}
}

@Article{DuchampHivertThibon2002,
  title={Noncommutative symmetric functions VI: free quasi-symmetric functions and related algebras},
  author={Duchamp, G{\'e}rard and Hivert, Florent and Thibon, Jean-Yves},
  journal={International Journal of Algebra and computation},
  volume={12},
  number={05},
  pages={671--717},
  year={2002},
  publisher={World Scientific}
}

@Book{FlajoletSedgewick2009,
 Author = {Philippe {Flajolet} and Robert {Sedgewick}},
 Title = {{Analytic combinatorics}},
 ISBN = {978-0-521-89806-5/hbk},
 Pages = {xiii + 810},
 Year = {2009},
 Publisher = {Cambridge: Cambridge University Press},
 Language = {English},
 MSC2010 = {05-02 05A15 05A16 60C05 60E10},
 Zbl = {1165.05001}
}

@Article{Gilbert1959,
 Author = {E. N. {Gilbert}},
 Title = {{Random graphs.}},
 FJournal = {{Annals of Mathematical Statistics}},
 Journal = {{Ann. Math. Stat.}},
 ISSN = {0003-4851},
 Volume = {30},
 Pages = {1141--1144},
 Year = {1959},
 Publisher = {Institute of Mathematical Statistics, Baltimore, MD},
 Language = {English},
 Zbl = {0168.40801}
}

@article{Klazar2003,
  title={Irreducible and connected permutations},
  author={Klazar, Martin},
  journal={Institut teoretick{\'e} informatiky (ITI) Series},
  number={122},
  year={2003}
}

@Article{KohRee2007,
 Author = {Koh, Youngmee and Ree, Sangwook},
 Title = {Connected permutation graphs},
 FJournal = {Discrete Mathematics},
 Journal = {Discrete Math.},
 ISSN = {0012-365X},
 Volume = {307},
 Number = {21},
 Pages = {2628--2635},
 Year = {2007},
 Language = {English},
 DOI = {10.1016/j.disc.2006.11.014},
 zbMATH = {5204993},
 Zbl = {1126.05058}
}

@Book{LandoZvonkin2004,
 Author = {Lando, S. K. and Zvonkin, A. K.},
 Title = {Graphs on surfaces and their applications. {Appendix} by {Don} {B}. {Zagier}},
 FSeries = {Encyclopaedia of Mathematical Sciences},
 Series = {Encycl. Math. Sci.},
 ISSN = {0938-0396},
 Volume = {141},
 ISBN = {3-540-00203-0},
 Year = {2004},
 Publisher = {Berlin: Springer},
 Language = {English},
 zbMATH = {2019766},
 Zbl = {1040.05001}
}

@Incollection{MonteilNurligareev2021,
  title={Asymptotics for Connected Graphs and Irreducible Tournaments},
  author={Monteil, Thierry and Nurligareev, Khaydar},
  booktitle={Extended Abstracts EuroComb 2021},
  pages={823--828},
  year={2021},
  publisher={Springer}
}

@Article{MoonMoser1962,
 Author = {J. W. {Moon} and L. {Moser}},
 Title = {{Almost all tournaments are irreducible}},
 FJournal = {{Canadian Mathematical Bulletin}},
 Journal = {{Can. Math. Bull.}},
 ISSN = {0008-4395; 1496-4287/e},
 Volume = {5},
 Pages = {61--65},
 Year = {1962},
 Publisher = {Cambridge University Press, Cambridge; Canadian Mathematical Society, Ottawa, ON},
 Language = {English},
 Zbl = {0105.33304}
}

@Misc{Moon1968,
 Author = {J. W. {Moon}},
 Title = {{Topics on tournaments}},
 Year = {1968},
 Language = {English},
 HowPublished = {{New York etc: Holt, Rinehart and Winston VIII,104 p 56 s, (1968).}},
 Zbl = {0191.22701}
}

@PHDTHESIS{Nurligareev2022,
  TITLE = {Irreducibility of combinatorial objects: asymptotic probability and interpretation},
  AUTHOR = {Nurligareev, Khaydar},
  URL = {http://www.theses.fr/2022PA131034/document},
  NUMBER = {2022PA131034},
  SCHOOL = {{Universit{\'e} Sorbonne Paris Nord — Paris XIII}},
  YEAR = {2022},
  MONTH = Oct,
  TYPE = {Theses},
  HAL_ID = {tel-03961144},
  HAL_VERSION = {v1},
  NOTE = "Thèse de doctorat dirigée par Thierry Monteil et Lionel Pournin, Informatique, Université Paris 13",
}

@InCollection{Odlyzko1995,
 Author = {Odlyzko, A. M.},
 Title = {Asymptotic enumeration methods},
 BookTitle = {Handbook of combinatorics. Vol. 1-2},
 ISBN = {0-444-88002-X; 0-444-82346-8; 0-444-82351-4; 0-262-07169-X; 0-262-07170-3; 0-262-07171-1},
 Pages = {1063--1229},
 Year = {1995},
 Publisher = {Amsterdam: Elsevier (North-Holland); Cambridge, MA: MIT Press},
 Language = {English},
 zbMATH = {863491},
 Zbl = {0845.05005}
}

@Article{Rado1943,
 Author = {R. {Rado}},
 Title = {{Theorems on linear combinatorial topology and general measure}},
 FJournal = {{Annals of Mathematics. Second Series}},
 Journal = {{Ann. Math. (2)}},
 ISSN = {0003-486X; 1939-8980/e},
 Volume = {44},
 Pages = {228--270},
 Year = {1943},
 Publisher = {Princeton University, Mathematics Department, Princeton, NJ},
 Language = {English},
 Zbl = {0061.09702}
}

@article{Renyi1959,
 Author = {R{\'e}nyi, Alfr{\'e}d},
 Title = {Some remarks on the theory of trees},
 FJournal = {Publications of the Mathematical Institute of the Hungarian Academy of Sciences},
 Journal = {Publ. Math. Inst. Hung. Acad. Sci.},
 Volume = {4},
 Pages = {73--85},
 Year = {1959},
 Language = {English},
 zbMATH = {3152798},
 Zbl = {0093.37604}
}

@Article{Roy1958,
 Author = {Bernard {Roy}},
 Title = {{Sur quelques propri\'et\'es des graphes fortement connexes}},
 FJournal = {{Comptes Rendus Hebdomadaires des S\'eances de l'Acad\'emie des Sciences, Paris}},
 Journal = {{C. R. Acad. Sci., Paris}},
 ISSN = {0001-4036},
 Volume = {247},
 Pages = {399--401},
 Year = {1958},
 Publisher = {Gauthier-Villars, Paris},
 Language = {French},
 Zbl = {0086.16302}
}

@Book{Stanley2012,
 Author = {Richard P. {Stanley}},
 Title = {{Enumerative combinatorics. Vol. 1.}},
 FJournal = {{Cambridge Studies in Advanced Mathematics}},
 Journal = {{Camb. Stud. Adv. Math.}},
 Volume = {49},
 ISBN = {978-1-107-60262-5; 978-1-107-01542-5; 978-1-139-20056-1},
 Pages = {xiii + 626},
 Year = {2012},
 Publisher = {Cambridge: Cambridge University Press},
 Language = {English},
 MSC2010 = {05-02 05A15 05A16 06A07},
 Zbl = {1247.05003}
}

@Book{Stanley1999,
 Author = {Richard P. {Stanley}},
 Title = {{Enumerative combinatorics. Vol. 2}},
 FJournal = {{Cambridge Studies in Advanced Mathematics}},
 Journal = {{Camb. Stud. Adv. Math.}},
 Volume = {62},
 ISBN = {0-521-56069-1},
 Pages = {xii + 581},
 Year = {1999},
 Publisher = {Cambridge: Cambridge University Press},
 Language = {English},
 MSC2010 = {05-02 05A15 05A16 06A07 05E05},
 Zbl = {0928.05001}
}

@Article{Wright1967,
 Author = {E. M. {Wright}},
 Title = {{A relationship between two sequences. I, II}},
 FJournal = {{Proceedings of the London Mathematical Society. Third Series}},
 Journal = {{Proc. Lond. Math. Soc. (3)}},
 ISSN = {0024-6115},
 Volume = {17},
 Pages = {296--304, 547--552},
 Year = {1967},
 Publisher = {John Wiley \& Sons, Chichester; London Mathematical Society, London},
 Language = {English},
 DOI = {10.1112/plms/s3-17.2.296},
 Zbl = {0147.31901}
}

@Article{Wright1968,
 Author = {Wright, E. M.},
 Title = {A relationship between two sequences. {III}},
 FJournal = {Journal of the London Mathematical Society},
 Journal = {J. Lond. Math. Soc.},
 ISSN = {0024-6107},
 Volume = {43},
 Pages = {720--724},
 Year = {1968},
 Language = {English},
 DOI = {10.1112/jlms/s1-43.1.720},
 zbMATH = {3256350},
 Zbl = {0159.25501}
}

@Article{Wright1970apr,
 Author = {E. M. {Wright}},
 Title = {{Asymptotic relations between enumerative functions in graph theory.}},
 FJournal = {{Proceedings of the London Mathematical Society. Third Series}},
 Journal = {{Proc. Lond. Math. Soc. (3)}},
 ISSN = {0024-6115; 1460-244X/e},
 Volume = {20},
 Pages = {558--572},
 Year = {1970},
 Publisher = {John Wiley \& Sons, Chichester; London Mathematical Society, London},
 Language = {English},
 Zbl = {0188.55901}
}

@Article{Wright1970jul,
 Author = {E. M. {Wright}},
 Title = {{The number of irreducible tournaments.}},
 FJournal = {{Glasgow Mathematical Journal}},
 Journal = {{Glasg. Math. J.}},
 ISSN = {0017-0895; 1469-509X/e},
 Volume = {11},
 Pages = {97--101},
 Year = {1970},
 Publisher = {Cambridge University Press, Cambridge},
 Language = {English},
 MSC2010 = {05C20},
 Zbl = {0209.03801}
}

@misc{MonteilNurligareevSET,
  title={Asymptotic probability for connectedness}, 
  author={Thierry Monteil and Khaydar Nurligareev},
  HowPublished = {{\url{https://arxiv.org/abs/2401.00818}}},
  year={2024},
  eprint={2401.00818},
  archivePrefix={arXiv},
  primaryClass={math.CO}
}

@Misc{MonteilNurligareevANTISEQ,
 Author = {Monteil, Thierry and Nurligareev, Khaydar},
 Title = {Anti-{SEQ}},
 HowPublished = {{In preparation}},
}

\

\

\appendix
\section{Numerical values}

Here, we provide numerical values of the asymptotic coefficients that appear in the applications discussed in this paper.
In each case, we precede the tables of coefficients with the counting sequences that comprise them.
A more detailed description of the content is given below.

\subsubsection*{Tournament and multitournament coefficients}

In Tables~\ref{table: counting sequences it_n^m} and~\ref{table: d_(k,m)(T)}, we indicate, respectively, the counting sequences $\big(\it_n^{(m)}\big)$ of labeled tournaments with $m$ irreducible parts and the asymptotic coefficients $\d_{k,m}(\mcT)$ that appear in Proposition~\ref{prop: SEQ_m-asymptotics for tournaments}.
In both cases, $m\le5$.

The following four tables are devoted to labeled multitournaments.
More precisely, in Tables~\ref{table: counting sequences it_n^m(2)} and~\ref{table: counting sequences it_n^m(3)} we expose the counting sequences $\big(\it_n^{(m)}(d)\big)$ for $d=2$ and $d=3$, respectively,
while Tables~\ref{table: d_(k,m)(T(2))} and~\ref{table: d_(k,m)(T(3))} show the corresponding asymptotic coefficients $\d_{k,m}\big(\mcT(d)\big)$ indicated in Proposition~\ref{prop: SEQ_m-asymptotics for multitournaments}.
In all cases, $m\le5$.

In principle, it is possible to represent these sequences as polynomials in $d$.
Thus, the sequence $\big(\it_n(d)\big)$ begins as follows:
\[
 1,\,\,
 (d-1),\,\,
 (d^3+3d^2-3d+1),\,\,
 (d^6+6d^5+15d^4+12d^3-15d^2+6d-1),\,\,
 \ldots
\]
while the first five terms of the sequence $\big(-\d_{k,1}(\mcT(d))\big)$ are
\[
 -1,\,\,
 2,\,\,
 2(d-2),\,\,
 2(d^3+3d^2-6d+4),\,\,
 2(d^6+6d^5+15d^4+8d^3-30d^2+24d-8),\,\,
 \ldots
\]

\subsubsection*{Permutation and multipermutation coefficients}

In Tables~\ref{table: counting sequences ip_n^m} and~\ref{table: d_(k,m)(P)}, we indicate, respectively, the counting sequences $\big(\ip_n^{(m)}\big)$ of permutations with $m$ irreducible parts and the asymptotic coefficients $\d_{k,m}(\mcP)$ that appear in Proposition~\ref{prop: SEQ_m-asymptotics for permutations},
both for $m\le5$.
In Tables~\ref{table: counting sequences ip_n^m(2)} and~\ref{table: counting sequences ip_n^m(3)}, we show the counting sequences~$\big(\ip_n^{(m)}(d)\big)$ if $d$-multipermutations with $m$ indecomposable parts for $d=2$ and $d=3$, respectively.
In Tables~\ref{table: d_(k,m)(P(2))} and~\ref{table: d_(k,m)(P(3))}, we show, for $d=2$ and $d=3$ respectively, the asymptotic coefficients $\d_{k,m}\big(\mcP(d)\big)$ indicated in relation~\eqref{formula: SEQ_m-asymptotics for multipermutations}.
Again, we consider $m\le5$ in all cases.

As functions of the parameter $d$, the first several numbers $\ip_n(d)$ of indecomposable $d$-multipermutations of size $n$, are
 \[
  1,\,\,
  (2^d-1),\,\,
  (6^d-2\cdot2^d+1),\,\,
  (24^d-2\cdot6^d-4^d+3\cdot2^d-1),\,\,
  \ldots
 \]
The sequence $\big(-\d_{k,1}(\mcP(d))\big)$ of asymptotic coefficients begins with
 \[
  2,\,\,
  (2\cdot2^d-3),\,\,
  (2\cdot6^{d}-6\cdot2^{d}+4),\,\,
  (2\cdot24^{d}-6\cdot6^d-3\cdot4^{d}+12\cdot2^d-5),\,\,
  \ldots
 \]

\subsubsection*{Perfect matching and multiple perfect matching coefficients}

In Tables~\ref{table: counting sequences im_2n^m} and~\ref{table: d_(2k,m)(M)}, we indicate, respectively, the counting sequences $\big(\im_{2n}^{(m)}\big)$ of perfect matchings with $m$ irreducible parts and the asymptotic coefficients $\d_{2k,m}(\mcM)$ that appear in Proposition~\ref{prop: SEQ_m-asymptotics for perfect matchings},
both for $m\le5$.
In Tables~\ref{table: counting sequences im_2n^m(2)} and~\ref{table: counting sequences im_2n^m(3)}, we show the counting sequences~$\big(\im_{2n}^{(m)}(d)\big)$ of $d$-multiple perfect matchings with $m$ irreducible parts for $d=2$ and $d=3$, respectively.
In Tables~\ref{table: d_(2k,m)(M(2))} and~\ref{table: d_(2k,m)(M(3))}, we show, for $d=2$ and $d=3$ respectively, the asymptotic coefficients $\d_{2k,m}\big(\mcM(d)\big)$ indicated in relation~\eqref{formula: SEQ_m-asymptotics for multiple perfect matchings}.
In all cases, $m\le5$.

As functions of the parameter $d$, the first several numbers $\im_{2n}(d)$ of indecomposable $d$-multiple perfect matchings of size $n$, are
 \[
  1,\,\,
  (3^d-1),\,\,
  (15^d-2\cdot3^d+1),\,\,
  (105^d-2\cdot15^d-9^d+3^{d+1}-1),\,\,
  \ldots
 \]
The sequence $\big(-\d_{2k,1}(\mcM(d))\big)$ of asymptotic coefficients begins with
 \[
  2,\,\,
  (2\cdot3^d-3),\,\,
  (2\cdot15^d-6\cdot3^d+4),\,\,
  (2\cdot105^d-6\cdot15^d-3\cdot9^d+12\cdot3^d-5),\,\,
  \ldots
 \]

\subsubsection*{Unlabeled tournament coefficients}

The last two tables are devoted to unlabeled tournaments.
In Table~\ref{table: counting sequences tilde(it)_n^m}, we expose the counting sequences $\big(\widetilde{\it}_n^{(m)}\big)$ of unlabeled tournaments with $m$ irreducible parts for $m\le5$.
Table~\ref{table: d_(k,m)(tilde(T))} shows the asymptotic coefficients $\d_{k,m}(\widetilde{\mcT})$ that appear in Proposition~\ref{prop: SEQ_m-asymptotics for tournaments, unlabeled}, also for $m\le5$.

\newpage


\begin{table}[ht!]
 \[
  \begin{array}{c|cccccccccc}
   n & 1 & 2 & 3 & 4 & 5 & 6 & 7 & 8 & 9 \\
   \hline
  \it_n & 1 & 0 & 2 & 24 & 544 & 22\,320 & 1\,677\,488 & 236\,522\,496 & 64\,026\,088\,576 & \ldots \\ 
  \it_n^{(2)} & 0 & 2 & 0 & 16 & 240 & 6\,608 & 315\,840 & 27\,001\,984 & 4\,268\,194\,560 & \ldots \\ 
  \it_n^{(3)} & 0 & 0 & 6 & 0 & 120 & 2\,160 & 70\,224 & 3\,830\,400 & 366\,729\,600 & \ldots \\ 
  \it_n^{(4)} & 0 & 0 & 0 & 24 & 0 & 960 & 20\,160 & 758\,016 & 46\,448\,640 & \ldots \\ 
  \it_n^{(5)} & 0 & 0 & 0 & 0 & 120 & 0 & 8\,400 & 201\,600 & 8\,628\,480 & \ldots \\ 
  \end{array}
 \]
 \caption{Counting sequences $\big(\it_n^{(m)}\big)$ for $m\le5$.}
 \label{table: counting sequences it_n^m}
\end{table}
 
\begin{table}[ht!]
 \[
  \begin{array}{c|cccccccccc}
   n & 0 & 1 & 2 & 3 & 4 & 5 & 6 & 7 & 8 & \\
   \hline
  \d_{k,1}(\mcT) & 1 & -2 & 2 & -4 & -32 & -848 & -38\,032 & -3\,039\,136 & -446\,043\,008 & \ldots \\ 
  \d_{k,2}(\mcT) & 0 & 2 & -8 & 16 & -16 & 368 & 22\,528 & 2\,232\,064 & 372\,697\,856 & \ldots \\ 
  \d_{k,3}(\mcT) & 0 & 0 & 6 & -36 & 120 & 0 & 9\,744 & 586\,656 & 60\,297\,600 & \ldots \\ 
  \d_{k,4}(\mcT) & 0 & 0 & 0 & 24 & -192 & 960 & 960 & 153\,216 & 10\,063\,872 & \ldots \\ 
  \d_{k,5}(\mcT) & 0 & 0 & 0 & 0 & 120 & -1\,200 & 8\,400 & 16\,800 & 2\,177\,280 & \ldots 
  \end{array}
 \]
 \caption{Asymptotic coefficients $\d_{k,m}(\mcT)$ for $m\le5$.}
 \label{table: d_(k,m)(T)}
\end{table}

 
\begin{table}[ht!]
 \[
  \begin{array}{c|ccccccccc}
   n & 1 & 2 & 3 & 4 & 5 & 6 & 7 & 8 \\
   \hline
  \it_n(2) & 1 & 1 & 15 & 543 & 51\,969 & 13\,639\,329 & 10\,259\,025\,615 & 22\,709\,334\,063\,807 & \ldots \\ 
  \it_n^{(2)}(2) & 0 & 2 & 6 & 126 & 5\,730 & 644\,418 & 193\,703\,454 & 165\,016\,159\,614 & \ldots \\ 
  \it_n^{(3)}(2) & 0 & 0 & 6 & 36 & 990 & 54\,360 & 6\,994\,134 & 2\,358\,537\,804 & \ldots \\ 
  \it_n^{(4)}(2) & 0 & 0 & 0 & 24 & 240 & 8\,280 & 534\,240 & 77\,136\,696 &  \ldots \\ 
  \it_n^{(5)}(2) & 0 & 0 & 0 & 0 & 120 & 1\,800 & 75\,600 & 5\,619\,600 & \ldots \\ 
  \end{array}
 \]
 \caption{Counting sequences $\big(\it_n^{(m)}(2)\big)$ for $m\le5$.}
 \label{table: counting sequences it_n^m(2)}
\end{table}

\begin{table}[ht!]
 \[
  \begin{array}{c|ccccccccc}
   n & 0 & 1 & 2 & 3 & 4 & 5 & 6 & 7 & \\
   \hline
  \d_{k,1}\big(\mcT(2)\big) & 1 & -2 & 0 & -24 & -960 & -98\,208 & -26\,634\,240 & -20\,324\,347\,776 & \ldots \\ 
  \d_{k,2}\big(\mcT(2)\big) & 0 & 2 & -6 & 18 & 654 & 82\,998 & 24\,809\,706 & 19\,757\,225\,682 & \ldots \\ 
  \d_{k,3}\big(\mcT(2)\big) & 0 & 0 & 6 & -18 & 234 & 11\,970 & 1\,631\,934 & 540\,748\,278 & \ldots \\ 
  \d_{k,4}\big(\mcT(2)\big) & 0 & 0 & 0 & 24 & -48 & 2\,520 & 158\,400 & 24\,005\,016 & \ldots \\ 
  \d_{k,5}\big(\mcT(2)\big) & 0 & 0 & 0 & 0 & 120 & 0 & 27\,000 & 1\,990\,800 & \ldots 
  \end{array}
 \]
 \caption{Asymptotic coefficients $\d_{k,m}\big(\mcT(2)\big)$ for $m\le5$.}
 \label{table: d_(k,m)(T(2))}
\end{table}
 
\begin{table}[ht!]
 \[
  \begin{array}{c|ccccccccc}
   n & 1 & 2 & 3 & 4 & 5 & 6 & 7 & 8 \\
   \hline
  \it_n(3) & 1 & 2 & 46 & 3\,608 & 1\,006\,936 & 1\,061\,010\,512 & 4\,382\,959\,945\,456 & \ldots \\ 
  \it_n^{(2)}(3) & 0 & 2 & 12 & 392 & 37\,920 & 12\,342\,032 & 14\,950\,347\,552 & \ldots \\ 
  \it_n^{(3)}(3) & 0 & 0 & 6 & 72 & 3\,120 & 358\,560 & 132\,424\,656 & \ldots \\ 
  \it_n^{(4)}(3) & 0 & 0 & 0 & 24 & 480 & 26\,400 & 3\,514\,560 & \ldots \\ 
  \it_n^{(5)}(3) & 0 & 0 & 0 & 0 & 120 & 3\,600 & 243\,600 & \ldots \\ 
  \end{array}
 \]
 \caption{Counting sequences $\big(\it_n^{(m)}(3)\big)$ for $m\le5$.}
 \label{table: counting sequences it_n^m(3)}
\end{table}
 
\begin{table}[ht!]
 \[
  \begin{array}{c|cccccccc}
   n & 0 & 1 & 2 & 3 & 4 & 5 & 6 & \\
   \hline
  \d_{k,1}\big(\mcT(3)\big) & 1 & -2 & -2 & -80 & -6\,824 &  -1\,975\,952 & -2\,109\,678\,992 & \ldots \\ 
  \d_{k,2}\big(\mcT(3)\big) & 0 & 2 & -4 & 56 & 5\,792 & 1\,868\,432 &  2\,073\,370\,016 & \ldots \\ 
  \d_{k,3}\big(\mcT(3)\big) & 0 & 0 & 6 & 0 & 816 & 96\,480 & 34\,953\,936 & \ldots \\ 
  \d_{k,4}\big(\mcT(3)\big) & 0 & 0 & 0 & 24 & 96 & 9\,120 & 1\,237\,440 & \ldots \\ 
  \d_{k,5}\big(\mcT(3)\big) & 0 & 0 & 0 & 0 & 120 & 1\,200 & 99\,600 & \ldots 
  \end{array}
 \]
 \caption{Asymptotic coefficients $\d_{k,m}\big(\mcT(3)\big)$ for $m\le5$.}
 \label{table: d_(k,m)(T(3))}
\end{table}


\begin{table}[ht!]
 \[
  \begin{array}{c|cccccccccccc}
   n & 1 & 2 & 3 & 4 & 5 & 6 & 7 & 8 & 9 & 10 & 11 \\
   \hline
  \ip_n & 1 & 1 & 3 & 13 & 71 & 461 & 3\,447 & 29\,093 & 273\,343 & 2\,829\,325 & 31\,998\,903 & \ldots \\
  \ip_n^{(2)} & 0 & 1 & 2 & 7 & 32 & 177 & 1\,142 & 8\,411 & 69\,692 & 642\,581 & 6\,534\,978 & \ldots \\
  \ip_n^{(3)} & 0 & 0 & 1 & 3 & 12 & 58 & 327 & 2\,109 & 15\,366 & 125\,316 & 1\,135\,329 & \ldots \\
  \ip_n^{(4)} & 0 & 0 & 0 & 1 & 4 & 18 & 92 & 531 & 3\,440 & 24\,892 & 200\,344 & \ldots \\
  \ip_n^{(5)} & 0 & 0 & 0 & 0 & 1 & 5 & 25 & 135 & 800 & 5\,226 & 37\,690 & \ldots
  \end{array}
 \]
 \caption{Counting sequences $\big(\ip_n^{(m)}\big)$ for $m\le5$.}
 \label{table: counting sequences ip_n^m}
\end{table}
 
\begin{table}[ht!]
 \[
  \begin{array}{c|ccccccccccc}
   n & 0 & 1 & 2 & 3 & 4 & 5 & 6 & 7 & 8 & 9 \\
   \hline
  \d_{k,1}(\mcP) & 1 & -2 & -1 & -4 & -19 & -110 & -745 & -5\,752 & -49\,775 & -476\,994 & \ldots \\ 
  \d_{k,2}(\mcP) & 0 & 2 & -2 & 0 & 4 & 38 & 330 & 2\,980 & 28\,760 & 298\,650 & \ldots \\ 
  \d_{k,3}(\mcP) & 0 & 0 & 3 & 0 & 6 & 36 & 237 & 1\,740 & 14\,172 & 127\,200 & \ldots \\ 
  \d_{k,4}(\mcP) & 0 & 0 & 0 & 4 & 4 & 20 & 108 & 672 & 4\,728 & 37\,144 & \ldots \\ 
  \d_{k,5}(\mcP) & 0 & 0 & 0 & 0 & 5 & 10 & 45 & 240 & 1\,470 & 10\,140 & \ldots 
  \end{array}
 \]
 \caption{Asymptotic coefficients $\d_{k,m}(\mcP)$ for $m\le5$.}
 \label{table: d_(k,m)(P)}
\end{table}


\begin{table}[ht!]
 \[
  \begin{array}{c|ccccccccccc}
   n & 1 & 2 & 3 & 4 & 5 & 6 & 7 & 8 \\
   \hline
  \ip_n(2) & 1 & 3 & 29 & 499 & 13\,101 & 486\,131 & 24\,266\,797 & 1\,571\,357\,619 & \ldots \\ 
  \ip_n^{(2)}(2) & 0 & 1 & 6 & 67 & 1\,172 & 30\,037 & 1\,079\,810 & 52\,459\,239 & \ldots \\ 
  \ip_n^{(3)}(2) & 0 & 0 & 1 & 9 & 114 & 2\,046 & 51\,591 & 1\,802\,079 & \ldots \\ 
  \ip_n^{(4)}(2) & 0 & 0 & 0 & 1 & 12 & 170 & 3\,148 & 78\,627 & \ldots \\ 
  \ip_n^{(5)}(2) & 0 & 0 & 0 & 0 & 1 & 15 & 235 & 4\,505 & \ldots \\ 
  \end{array}
 \]
 \caption{Counting sequences $\big(\ip_n^{(m)}(2)\big)$ for $m\le5$.}
 \label{table: counting sequences ip_n^m(2)}
\end{table}

\begin{table}[ht!]
 \[
  \begin{array}{c|ccccccccc}
   n & 0 & 1 & 2 & 3 & 4 & 5 & 6 & 7 \\
   \hline
  \d_{k,1}\big(\mcP(2)\big) & 1 & -2 & -5 & -52 & -931 & -25\,030 & -942\,225 & -47\,453\,784 & \ldots \\ 
  \d_{k,2}\big(\mcP(2)\big) & 0 & 2 & 2 & 36 & 748 & 21\,742 & 856\,206 & 44\,317\,536 & \ldots \\ 
  \d_{k,3}\big(\mcP(2)\big) & 0 & 0 & 3 & 12 & 150 & 2\,868 & 78\,345 & 2\,939\,328 & \ldots \\ 
  \d_{k,4}\big(\mcP(2)\big) & 0 & 0 & 0 & 4 & 28 & 364 & 6\,884 & 182\,120 & \ldots \\ 
  \d_{k,5}\big(\mcP(2)\big) & 0 & 0 & 0 & 0 & 5 & 50 & 705 & 13\,480 & \ldots 
  \end{array}
 \]
 \caption{Asymptotic coefficients $\d_{k,m}\big(\mcP(2)\big)$ for $m\le5$.}
 \label{table: d_(k,m)(P(2))}
\end{table}

\begin{table}[ht!]
 \[
  \begin{array}{c|ccccccccccc}
   n & 1 & 2 & 3 & 4 & 5 & 6 & 7 \\
   \hline
  \ip_n(3) & 1 & 7 & 201 & 13\,351 & 1\,697\,705 & 369\,575\,303 & 127\,249\,900\,617 & \ldots \\ 
  \ip_n^{(2)}(3) & 0 & 1 & 14 & 451 & 29\,516 & 3\,622\,725 & 768\,285\,578 & \ldots \\ 
  \ip_n^{(3)}(3) & 0 & 0 & 1 & 21 & 750 & 48\,838 & 5\,804\,607 & \ldots \\ 
  \ip_n^{(4)}(3) & 0 & 0 & 0 & 1 & 28 & 1\,098 & 71\,660 & \ldots \\ 
  \ip_n^{(5)}(3) & 0 & 0 & 0 & 0 & 1 & 35 & 1\,495 & \ldots 
  \end{array}
 \]
 \caption{Counting sequences $\big(\ip_n^{(m)}(3)\big)$ for $m\le5$.}
 \label{table: counting sequences ip_n^m(3)}
\end{table}

\begin{table}[ht!]
 \[
  \begin{array}{c|cccccccc}
   n & 0 & 1 & 2 & 3 & 4 & 5 & 6 \\
   \hline
  \d_{k,1}\big(\mcP(3)\big) & 1 & -2 & -13 & -388 & -26\,251 & -3\,365\,894 & -735\,527\,881 & \ldots \\ 
  \d_{k,2}\big(\mcP(3)\big) & 0 & 2 & 10 & 348 & 24\,940 & 3\,278\,846 & 724\,757\,382 & \ldots \\ 
  \d_{k,3}\big(\mcP(3)\big) & 0 & 0 & 3 & 36 & 1\,230 & 84\,132 & 10\,578\,441 & \ldots \\ 
  \d_{k,4}\big(\mcP(3)\big) & 0 & 0 & 0 & 4 & 76 & 2\,780 & 186\,708 & \ldots \\ 
  \d_{k,5}\big(\mcP(3)\big) & 0 & 0 & 0 & 0 & 5 & 130 & 5\,145 & \ldots 
  \end{array}
 \]
 \caption{Asymptotic coefficients $\d_{k,m}\big(\mcP(3)\big)$ for $m\le5$.}
 \label{table: d_(k,m)(P(3))}
\end{table}


\begin{table}[ht!]
 \[
  \begin{array}{c|ccccccccccc}
   n & 1 & 2 & 3 & 4 & 5 & 6 & 7 & 8 & 9 & 10 \\
   \hline
  \im_{2n} & 1 & 2 & 10 & 74 & 706 & 8\,162 & 110\,410 & 1\,708\,394 & 29\,752\,066 & 576\,037\,442 & \ldots \\ 
  \im_{2n}^{(2)} & 0 & 1 & 4 & 24 & 188 & 1\,808 & 20\,628 & 273\,064 & 4\,126\,156 & 70\,252\,320 & \ldots \\ 
  \im_{2n}^{(3)} & 0 & 0 & 1 & 6 & 42 & 350 & 3\,426 & 38\,886 & 506\,314 & 7\,491\,006 & \ldots \\ 
  \im_{2n}^{(4)} & 0 & 0 & 0 & 1 & 8 & 64 & 568 & 5\,696 & 64\,744 & 833\,280 & \ldots \\ 
  \im_{2n}^{(5)} & 0 & 0 & 0 & 0 & 1 & 10 & 90 & 850 & 8\,770 & 100\,362 & \ldots 
  \end{array}
 \]
 \caption{Counting sequences $\big(\im_{2n}^{(m)}\big)$ for $m\le5$.}
 \label{table: counting sequences im_2n^m}
\end{table}
 
\begin{table}[ht!]
 \[
  \begin{array}{c|cccccccccc}
   n & 0 & 1 & 2 & 3 & 4 & 5 & 6 & 7 & 8 \\
   \hline
  \d_{2k,1}(\mcM) & 1 & -2 & -3 & -16 & -124 & -1\,224 & -14\,516 & -200\,192 & -3\,143\,724 & \ldots \\ 
  \d_{2k,2}(\mcM) & 0 & 2 & 0 & 6 & 64 & 744 & 9\,792 & 145\,160 & 2\,402\,304 & \ldots \\ 
  \d_{2k,3}(\mcM) & 0 & 0 & 3 & 6 & 39 & 336 & 3\,516 & 43\,032 & 602\,964 & \ldots \\ 
  \d_{2k,4}(\mcM) & 0 & 0 & 0 & 4 & 16 & 108 & 928 & 9\,520 & 113\,376 & \ldots \\ 
  \d_{2k,5}(\mcM) & 0 & 0 & 0 & 0 & 5 & 30 & 225 & 2\,000 & 20\,580 & \ldots 
  \end{array}
 \]
 \caption{Asymptotic coefficients $\d_{2k,m}(\mcM)$ for $m\le5$.}
 \label{table: d_(2k,m)(M)}
\end{table}


\begin{table}[ht!]
 \[
  \begin{array}{c|cccccccc}
   n & 1 & 2 & 3 & 4 & 5 & 6 & 7 \\
   \hline
  \im_{2n}(2) & 1 & 8 & 208 & 10\,520 & 867\,808 & 106\,065\,512 & 18\,027\,732\,016 & \ldots \\ 
  \im_{2n}^{(2)}(2) & 0 & 1 & 16 & 480 & 24\,368 & 1\,947\,200 & 230\,392\,272 & \ldots \\ 
  \im_{2n}^{(3)}(2) & 0 & 0 & 1 & 24 & 816 & 42\,056 & 3\,278\,112 & \ldots \\ 
  \im_{2n}^{(4)}(2) & 0 & 0 & 0 & 1 & 32 & 1\,216 & 64\,096 & \ldots \\ 
  \im_{2n}^{(5)}(2) & 0 & 0 & 0 & 0 & 1 & 40 & 1\,680 & \ldots 
  \end{array}
 \]
 \caption{Counting sequences $\big(\im_{2n}^{(m)}(2)\big)$ for $m\le5$.}
 \label{table: counting sequences im_2n^m(2)}
\end{table}
 
\begin{table}[ht!]
 \[
  \begin{array}{c|cccccccc}
   n & 0 & 1 & 2 & 3 & 4 & 5 & 6 \\
   \hline
  \d_{2k,1}\big(\mcM(2)\big) & 1 & -2 & -15 & -400 & -20\,560 & -1\,711\,248 & -210\,183\,824 & \ldots \\ 
  \d_{2k,2}\big(\mcM(2)\big) & 0 & 2 & 12 & 354 & 19\,168 & 1\,639\,776 & 204\,426\,336 & \ldots \\ 
  \d_{2k,3}\big(\mcM(2)\big) & 0 & 0 & 3 & 42 & 1\,299 & 68\,304 & 5\,592\,912 & \ldots \\ 
  \d_{2k,4}\big(\mcM(2)\big) & 0 & 0 & 0 & 4 & 88 & 3\,012 & 158\,656 & \ldots \\ 
  \d_{2k,5}\big(\mcM(2)\big) & 0 & 0 & 0 & 0 & 5 & 150 & 5\,685 & \ldots 
  \end{array}
 \]
 \caption{Asymptotic coefficients $\d_{2k,m}(\mcM(2))$ for $m\le5$.}
 \label{table: d_(2k,m)(M(2))}
\end{table}

\begin{table}[ht!]
\small
 \[
  \begin{array}{c|cccccccc}
   n & 1 & 2 & 3 & 4 & 5 & 6 & 7 \\
   \hline
  \im_{2n}(3) & 1 & 26 & 3\,322 & 1\,150\,226 & 841\,423\,330 & 1\,121\,484\,681\,818 & 2\,465\,466\,393\,826\,522 & \ldots \\ 
  \im_{2n}^{(2)}(3) & 0 & 1 & 52 & 7\,320 & 2\,473\,196 & 1\,753\,694\,096 & 2\,294\,365\,478\,340 & \ldots \\ 
  \im_{2n}^{(3)}(3) & 0 & 0 & 1 & 78 & 11\,994 & 3\,986\,486 & 2\,743\,549\,314 & \ldots \\ 
  \im_{2n}^{(4)}(3) & 0 & 0 & 0 & 1 & 104 & 17\,344 & 5\,707\,672 & \ldots \\ 
  \im_{2n}^{(5)}(3) & 0 & 0 & 0 & 0 & 1 & 130 & 23\,370 & \ldots 
  \end{array}
 \]
 \caption{Counting sequences $\big(\im_{2n}^{(m)}(3)\big)$ for $m\le5$.}
 \label{table: counting sequences im_2n^m(3)}
\end{table}
 
\begin{table}[ht!]
 \small
 \[
  \begin{array}{c|cccccccc}
   n & 0 & 1 & 2 & 3 & 4 & 5 & 6 \\
   \hline
  \d_{2k,1}\big(\mcM(3)\big) & 1 & -2 & -51 & -6\,592 & -2\,293\,132 & -1\,680\,373\,464 & -2\,241\,215\,669\,540 & \ldots \\ 
  \d_{2k,2}\big(\mcM(3)\big) & 0 & 2 & 48 & 6\,438 & 2\,271\,328 & 1\,672\,977\,864 & 2\,235\,962\,560\,224 & \ldots \\ 
  \d_{2k,3}\big(\mcM(3)\big) & 0 & 0 & 3 & 150 & 21\,495 & 7\,347\,936 & 5\,237\,215\,404 & \ldots \\ 
  \d_{2k,4}\big(\mcM(3)\big) & 0 & 0 & 0 & 4 & 304 & 47\,148 & 15\,807\,712 & \ldots \\ 
  \d_{2k,5}\big(\mcM(3)\big) & 0 & 0 & 0 & 0 & 5 & 510 & 85\,425 & \ldots 
  \end{array}
 \]
 \caption{Asymptotic coefficients $\d_{2k,m}(\mcM(3))$ for $m\le5$.}
 \label{table: d_(2k,m)(M(3))}
\end{table}


\begin{table}[ht!]
 \[
  \begin{array}{c|cccccccccccc}
   n & 1 & 2 & 3 & 4 & 5 & 6 & 7 & 8 & 9 & 10 & 11 \\
   \hline
  \widetilde{\it}_n & 1 & 0 & 1 & 1 & 6 & 35 & 353 & 6\,008 & 178\,133 & 9\,355\,949 & 884\,464\,590 & \ldots \\ 
  \widetilde{\it}_n^{(2)} & 0 & 1 & 0 & 2 & 2 & 13 & 72 & 719 & 12\,098 & 357\,078 & 18\,725\,040 & \ldots \\ 
  \widetilde{\it}_n^{(3)} & 0 & 0 & 1 & 0 & 3 & 3 & 21 & 111 & 1\,099 & 18\,273 & 536\,856 & \ldots \\ 
  \widetilde{\it}_n^{(4)} & 0 & 0 & 0 & 1 & 0 & 4 & 4 & 30 & 152 & 1\,494 & 24\,536 & \ldots \\ 
  \widetilde{\it}_n^{(5)} & 0 & 0 & 0 & 0 & 1 & 0 & 5 & 5 & 40 & 195 & 1\,905 & \ldots \\ 
  \end{array}
 \]
 \caption{Counting sequences $\big(\widetilde{\it}_n^{(m)}\big)$ for $m\le5$.}
 \label{table: counting sequences tilde(it)_n^m}
\end{table}
 
\begin{table}[ht!]
 \[
  \begin{array}{c|ccccccccccc}
   n & 0 & 1 & 2 & 3 & 4 & 5 & 6 & 7 & 8 & 9 & \\
   \hline
  \d_{k,1}(\widetilde{\mcT}) & 1 & -2 & 1 & -2 & 0 & -10 & -57 & -634 & -11\,297 & -344\,168 & \ldots \\ 
  \d_{k,2}(\widetilde{\mcT}) & 0 & 2 & -4 & 4 & -6 & 10 & 24 & 460 & 9\,362 & 310\,072 & \ldots \\ 
  \d_{k,3}(\widetilde{\mcT}) & 0 & 0 & 3 & -6 & 9 & -12 & 33 & 102 & 1\,581 & 30\,156 & \ldots \\ 
  \d_{k,4}(\widetilde{\mcT}) & 0 & 0 & 0 & 4 & -8 & 16 & -20 & 72 & 224 & 3\,340 & \ldots \\ 
  \d_{k,5}(\widetilde{\mcT}) & 0 & 0 & 0 & 0 & 5 & -10 & 25 & -30 & 130 & 390 & \ldots 
  \end{array}
 \]
 \caption{Asymptotic coefficients $\d_{k,m}(\widetilde{\mcT})$ for $m\le5$.}
 \label{table: d_(k,m)(tilde(T))}
\end{table}

\end{document}